\documentclass[11pt,reqno]{amsart}
\usepackage[margin=1.1in]{geometry}

\usepackage[T1]{fontenc}
\usepackage[utf8]{inputenc}
\usepackage{lmodern}
\usepackage{microtype}
\usepackage{mathtools}
\usepackage{amssymb}
\usepackage{amsmath}
\usepackage{amsthm}
\usepackage{mathrsfs}
\usepackage{xcolor}
\usepackage[colorlinks=true,linkcolor=blue!50!black,citecolor=blue!50!black,urlcolor=blue!75!black]{hyperref}
\usepackage[all]{xy}

\newtheorem{theorem}{Theorem}[section]
\newtheorem{proposition}[theorem]{Proposition}
\newtheorem{corollary}[theorem]{Corollary}
\newtheorem{lemma}[theorem]{Lemma}
\theoremstyle{definition}
\newtheorem{definition}[theorem]{Definition}
\newtheorem{example}[theorem]{Example}
\theoremstyle{remark}
\newtheorem{remark}[theorem]{Remark}

\DeclareMathOperator{\Rep}{Rep}
\DeclareMathOperator{\Fl}{Fl}
\DeclareMathOperator{\Gr}{Gr}
\DeclareMathOperator{\Hom}{Hom}
\DeclareMathOperator{\Ext}{Ext}
\DeclareMathOperator{\ext}{ext}

\DeclareMathOperator{\GL}{GL}
\DeclareMathOperator{\SI}{SI}
\DeclareMathOperator{\rk}{rk}
\DeclareMathOperator{\codim}{codim}
\DeclareMathOperator{\Tot}{Tot}
\DeclareMathOperator{\gr}{gr}
\DeclareMathOperator{\fext}{fext}
\DeclareMathOperator{\Sym}{Sym}

\newcommand{\CC}{\mathbb C}
\newcommand{\NN}{\mathbb N}
\newcommand{\ZZ}{\mathbb Z}
\newcommand{\kk}{\Bbbk}
\newcommand{\cA}{\mathcal A}
\newcommand{\cE}{\mathcal E}

\newcommand{\cG}{\mathcal G}
\newcommand{\cH}{\mathcal H}
\newcommand{\cL}{\mathcal L}
\newcommand{\cO}{\mathcal O}
\newcommand{\cP}{\mathcal P}
\newcommand{\cQ}{\mathcal Q}
\newcommand{\cS}{\mathcal S}

\newcommand{\bI}{\mathbf I}
\newcommand{\bM}{\mathbf M}
\newcommand{\bN}{\mathbf N}
\newcommand{\bU}{\mathbf U}
\newcommand{\bV}{\mathbf V}
\newcommand{\bW}{\mathbf W}
\newcommand{\bw}{\mathbf w}
\newcommand{\bv}{\mathbf v}
\newcommand{\bbeta}{\boldsymbol\beta}
\newcommand{\bdelta}{\boldsymbol\delta}
\newcommand{\blambda}{\boldsymbol\lambda}
\newcommand{\bnu}{\boldsymbol\nu}
\newcommand{\bmu}{\boldsymbol\mu}
\newcommand{\comp}{\mathsf c}
\newcommand{\Schur}{\mathbb S}

\newcommand{\topchi}{\chi_{\mathrm{top}}}
\newcommand{\udHom}{\underline{\Hom}}

\title[Flags of subrepresentations of general quiver representations]
{Flag Varieties of General Quiver Representations:\\Characteristic Classes and Semi-Invariants}
\author{Jiarui Fei}
\thanks{The author was supported in part by National Natural Science Foundation of China (No. 12131015 and No. 12571038).}
\address{School of Mathematical Sciences, Shanghai Jiao Tong University}
\email{jiarui@sjtu.edu.cn}
\date{}

\subjclass[2020]{Primary 14C17,16G20; Secondary 14M15, 14N15}
\keywords{quiver representation, flag of subrepresentations, quiver Grassmannian, cluster character, Chern class, Euler characteristic, equivariant localization, semi-invariant}

\hypersetup{
  pdftitle={Flag Varieties of General Quiver Representations: Characteristic Classes and Semi-Invariants},
  pdfauthor={Jiarui Fei},
  pdfsubject={Flags of subrepresentations of general quiver representations, characteristic classes, and quiver semi-invariants},
  pdfkeywords={quiver representations, flags of subrepresentations, quiver Grassmannians, cluster characters, Schubert calculus, Chern classes, semi-invariants}
}

\begin{document}

\begin{abstract}
We study flags of subrepresentations of a general representation of a finite quiver.
For each representation, its flag scheme is the zero scheme of a
section of a quotient bundle on a product of partial flag varieties.
We obtain the generic dimension and cycle class and equivariant localization
formulas for Euler characteristics and parabolic Schubert coefficients.
We further express these coefficients as finite signed
sums of semi-invariant dimensions for auxiliary star quivers with inward and
outward flag arms.  In the acyclic one-step case, this is compatible with the
result of Derksen--Schofield--Weyman, and the zero-dimensional case recovers
their semi-invariant counting theorem.  Examples treat matrix pencils and
generalized Kronecker quivers.
\end{abstract}

\maketitle

\section*{Introduction}\label{sec:introduction}

Let $Q=(Q_0,Q_1)$ be a finite quiver; loops and multiple arrows are allowed.  For a
dimension vector $\alpha\in\NN^{Q_0}$, write
\[ \Rep(Q,\alpha)=\bigoplus_{a\in Q_1} \Hom\bigl(\kk^{\alpha(t(a))},\kk^{\alpha(h(a))}\bigr). \]
Throughout, \emph{general} means belonging to a suitable nonempty Zariski-open
subset.  Over $\CC$, $\topchi$ denotes the compactly supported topological Euler
characteristic; for the projective varieties occurring here it is the usual topological
Euler characteristic.  Fix a \emph{flag type}
\begin{equation}\label{eq:intro-flag-type}
 0=\beta^0\leq\beta^1\leq\cdots\leq\beta^m
 \leq\beta^{m+1}=\alpha,
\end{equation}
and put $\delta^i=\beta^i-\beta^{i-1}$.  For
$V\in\Rep(Q,\alpha)$, let $\Fl_{\bbeta}(V)$ be the scheme of chains
\[ 0=U^0\subset U^1\subset\cdots\subset U^m\subset U^{m+1}=V, \qquad \dim U^k=\beta^k, \]
in which every $U^k$ is a subrepresentation.  The one-step case is the quiver
Grassmannian $\Gr_{\beta^1}(V)$.  For acyclic $Q$, its Euler characteristics are
the coefficients of the quiver-Grassmannian generating polynomial that enters
Caldero--Chapoton-type cluster characters.  With principal coefficients this is the
corresponding $F$-polynomial, and its generic values occur in Dupont's generic
character \cite{CC,Dupont}.

A \emph{quiver flag variety} here means a flag variety of subrepresentations;
the same term is also used for the GIT moduli spaces in \cite{Craw}.  For the varieties considered here, Wolf developed the scheme-theoretic
incidence construction, tangent-space calculation, dimension estimates, and reflection
functors \cite{Wolf}.  Cerulli Irelli, Feigin, and Reineke identified type-$A$
degenerate flag varieties with quiver Grassmannians and studied a broader class arising
from sums of projective and injective Dynkin representations, proving normal
local-complete-intersection and cellularity results \cite{CFR}.  Haupt computed Euler
characteristics for flags attached to tree and band modules \cite{Haupt}.  Related
Hall-algebra counting methods and quiver
determinantal varieties appear in \cite{FeiHall,Feidet}.  Maksimau proved
odd-cohomology vanishing in Dynkin
type, affine pavings in types $A$ and $D$, and a diagonal decomposition for rigid
representations \cite{Maksimau}; Zhou established affine pavings in all Dynkin types and
in several affine-type cases \cite{Zhou}.  In the one-step setting, Baldoni, Vergne,
and Walter characterize the Schubert positions of subrepresentations of a general
quiver representation by Horn-type conditions \cite{BVW}.  
Wolf's work already treats universal incidence varieties and filtered tangent
and extension spaces.  
In the one-step case, the cycle-class formula is due to
Crawley--Boevey \cite{CB}, and the semi-invariant counting theorem and the
positive-dimensional covariant interpretation are due to
Derksen--Schofield--Weyman \cite{DSW}.
In this paper we fix $(Q,\alpha,\bbeta)$ and vary $V$ in
$\Rep(Q,\alpha)$ to compute the cycle class and characteristic numbers of the
general quiver flag variety in a fixed product of partial flag varieties.

We work on the product of partial flag varieties
\[ X_{\bbeta}=\prod_{x\in Q_0} \Fl\bigl(\beta^1(x),\ldots,\beta^m(x);\alpha(x)\bigr). \]
Write $\cG_x^i=\cS_x^i/\cS_x^{i-1}$ for the graded pieces of the universal flags. For an arrow $a:x\to y$, let $\cP_a$
be the bundle of filtration-preserving maps and define
\[ \cE_a= \udHom\bigl(\kk^{\alpha(x)},\kk^{\alpha(y)}\bigr)/\cP_a, \qquad \cE_{\bbeta}=\bigoplus_{a\in Q_1}\cE_a. \]
Its associated graded is
\[ \gr\cE_a\cong \bigoplus_{1\leq i<j\leq m+1} (\cG_x^i)^\vee\otimes\cG_y^j. \]
To each representation $V$ we associate a section $s_V$ of
$\cE_{\bbeta}$ whose zero scheme is $\Fl_{\bbeta}(V)$.

Write
\[ \langle\eta,\theta\rangle_Q =\sum_x\eta(x)\theta(x) -\sum_{a\in Q_1}\eta(t(a))\theta(h(a)) \]
for the Euler form, and set
\[ d_{\bbeta}=\sum_{i<j}\langle\delta^i,\delta^j\rangle_Q, \qquad r_{\bbeta}=\rk\cE_{\bbeta}. \]
We state three main results.

\medskip
\noindent
\textbf{Theorem A (generic geometry).}
The universal incidence variety is the total space of the bundle of
filtration-preserving representations over $X_{\bbeta}$.  If its projection to
$\Rep(Q,\alpha)$ is dominant, then for general $V$ the section $s_V$ is regular and
\[ \dim\Fl_{\bbeta}(V)=d_{\bbeta}, \qquad [\Fl_{\bbeta}(V)] =c_{r_{\bbeta}}(\cE_{\bbeta})\cap[X_{\bbeta}]. \]
The top Chern class is nonzero exactly in the dominant case; hence generic existence is
independent of the algebraically closed ground field.  In characteristic zero the
general fiber is smooth.  The ladder algebra
$\Lambda=\kk Q\otimes_\kk\kk A_{m+1}$ gives
\[ T_U\Fl_{\bbeta}(V)\cong\Hom_\Lambda(\bU,\bV/\bU), \qquad \dim T_U\Fl_{\bbeta}(V) =d_{\bbeta}+\dim\Ext^1_\Lambda(\bU,\bV/\bU), \]
and dominance is equivalent in characteristic zero to vanishing of the generic
filtered $\Ext^1$.  See Theorem~\ref{thm:generic-regular-zero-locus},
Corollary~\ref{cor:chern-dominance}, and
Theorem~\ref{thm:filtered-ext-criterion}.

\medskip
\noindent
\textbf{Theorem B (characteristic numbers).}
Over $\CC$, let $Y=\Fl_{\bbeta}(V)$ be a general nonempty fiber.  For every homogeneous
Chern polynomial $P$ of weighted degree $d_{\bbeta}$,
\[ \int_Y P(c(T_Y)) =\int_{X_{\bbeta}}c_{r_{\bbeta}}(\cE_{\bbeta}) P\bigl(c(T_{X_{\bbeta}}-\cE_{\bbeta})\bigr). \]
In particular,
\[ \topchi(Y)= \int_{X_{\bbeta}} c_{r_{\bbeta}}(\cE_{\bbeta}) \left[\frac{c(T_{X_{\bbeta}})}{c(\cE_{\bbeta})}\right]_{d_{\bbeta}}. \]
The filtrations of $T_{X_{\bbeta}}$ and $\cE_{\bbeta}$ yield explicit Chern-root
products, and equivariant localization on $X_{\bbeta}$ expresses their integrals as
finite sums indexed by ordered set partitions. See
Theorems~\ref{thm:flag-euler-integral} and \ref{thm:flag-localization}.

\medskip
\noindent
\textbf{Theorem C (Schubert coefficients and semi-invariants).}
Let $\{\sigma_{\bw}\}$ be the product parabolic Schubert basis and define
\[ N_{\bw}(Q,\alpha,\bbeta) =\int_{X_{\bbeta}}c_{r_{\bbeta}}(\cE_{\bbeta})\sigma_{\bw}, \qquad \ell(\bw)=d_{\bbeta}. \]
We call these nonnegative integers the \emph{Schubert coefficients of the generic flag cycle}.
The \emph{twisted exterior class} of an auxiliary acyclic quiver, called the
\emph{layered quiver}, gives a first expression for $N_{\bw}$ in terms of
\emph{virtual Schur coefficients}. Rectangular Cauchy and
Littlewood--Richardson identities then give
\[ N_{\bw} =\sum_{\blambda,\bnu} (-1)^{|\blambda|}\mathsf b_{\bw,\bnu} \dim\SI(\widehat Q_{\bdelta},\widehat\delta)_ {\widehat\sigma(\blambda,\bnu)}. \]
Here $(\widehat Q_{\bdelta},\widehat\delta)$ is a generally disconnected
auxiliary acyclic quiver whose components are local star quivers with inward and outward
flag arms.  It is fixed by $(Q,\alpha,\bbeta)$, and only the weight varies.  The
formula is signed and does not express $N_{\bw}$ as a single semi-invariant
multiplicity.
For one-step flags and acyclic $Q$, the coefficients also admit the
nonnegative covariant formula of Derksen--Schofield--Weyman, as recalled in
Subsection~\ref{subsec:one-step-stronger}; this is a separate, stronger
interpretation rather than a termwise specialization of the signed formula.
For triple flags, Corollary~\ref{cor:LR-localization} gives a finite
fixed-point formula for Littlewood--Richardson coefficients. See
Theorems~\ref{thm:virtual-schur-Nw} and \ref{thm:auxiliary-SI-Nw}.

We first study flags for a general two-arrow matrix pencil.  This family has
dimension $2(n-2)$ and Euler characteristic $2^{n-3}n!$; its first member is a
degree-six del Pezzo surface.  A related one-step quiver Grassmannian gives a
projective small resolution of a six-nodal intersection of two quadrics. The latter example also exhibits the Littlewood--Richardson expansion
of a nontrivial Schubert coefficient.
We then give an all-parameter Chern-root
formula for generalized Kronecker quivers, determinantal plane curves, and Calabi--Yau complete
intersections.  For the Kronecker families, the Euler characteristics are terms in
quiver-Grassmannian $F$-polynomials.

Sections~2--3 establish Theorem~A, and Sections~4--5 establish Theorem~B.
Section~6 introduces the Schubert coefficients and their pairing formula;
Sections~7--8 prove Theorem~C.  Examples appear in Section~9.  The appendices
record the divided-difference formula for Schubert pairings and the arm
recurrences used in the semi-invariant construction.

\section{The universal equation bundle}

Throughout this section, $\kk$ is an algebraically closed field of arbitrary
characteristic.  Fix $Q$, $\alpha$, and a flag type \eqref{eq:intro-flag-type}.  Put
\[ X=X_{\bbeta}= \prod_{x\in Q_0}X_x, \qquad X_x=\Fl\bigl(\beta^1(x),\ldots,\beta^m(x);\alpha(x)\bigr). \]
On $X_x$, and hence on $X$ by pullback, let
\[ 0=\cS_x^0\subset\cS_x^1\subset\cdots\subset\cS_x^m \subset\cS_x^{m+1}=\cA_x :=\kk^{\alpha(x)}\otimes\cO_X \]
be the universal filtration.  Set
\[ \cG_x^i=\cS_x^i/\cS_x^{i-1}, \qquad \rk\cG_x^i=\delta^i(x). \]

For an arrow $a:x\to y$, let
\[ \cH_a=\udHom(\cA_x,\cA_y) =\Hom\bigl(\kk^{\alpha(x)},\kk^{\alpha(y)}\bigr)\otimes\cO_X. \]
There is a natural morphism
\begin{equation}\label{eq:rho-a}
 \rho_a:\cH_a\longrightarrow
 \bigoplus_{k=1}^m
 \udHom\bigl(\cS_x^k,\cA_y/\cS_y^k\bigr),
\end{equation}
whose $k$th component restricts a map to $\cS_x^k$ and then projects modulo
$\cS_y^k$.

\begin{definition}\label{def:P-E}
The \emph{filtration-preserving bundle} and the \emph{equation bundle} for $a$ are
\[ \cP_a=\ker\rho_a, \qquad \cE_a=\cH_a/\cP_a. \]
We set
\[ \cP_{\bbeta}=\bigoplus_{a\in Q_1}\cP_a, \qquad \cE_{\bbeta}=\bigoplus_{a\in Q_1}\cE_a. \]
\end{definition}

\begin{lemma}\label{lem:block-filtration}
For every arrow $a:x\to y$, the sheaves $\cP_a$ and $\cE_a$ are vector bundles.
They admit filtrations whose associated graded bundles are
\begin{align}
 \gr\cP_a
 &\cong
 \bigoplus_{1\leq j\leq i\leq m+1}
 (\cG_x^i)^\vee\otimes\cG_y^j,
 \label{eq:gr-P}\\
 \gr\cE_a
 &\cong
 \bigoplus_{1\leq i<j\leq m+1}
 (\cG_x^i)^\vee\otimes\cG_y^j.
 \label{eq:gr-E}
\end{align}
In particular,
\begin{align*}
 \rk\cP_a=\sum_{1\leq j\leq i\leq m+1}
 \delta^i(x)\delta^j(y),
 \qquad
 \rk\cE_a=\sum_{1\leq i<j\leq m+1}
 \delta^i(x)\delta^j(y).
\end{align*}
\end{lemma}

\begin{proof}
Work first on an open set on which the universal filtrations at $x$ and $y$ are
split.  Thus
\[
 \cA_x\cong\bigoplus_i\cG_x^i,
 \qquad
 \cA_y\cong\bigoplus_j\cG_y^j,
 \qquad
 \cS_x^k\cong\bigoplus_{i\leq k}\cG_x^i,
 \qquad
 \cS_y^k\cong\bigoplus_{j\leq k}\cG_y^j.
\]
A block
$\phi_{ji}:\cG_x^i\to\cG_y^j$ preserves all steps of the filtrations if and only if
$j\leq i$. Indeed, the source graded piece $\cG_x^i$ lies in every $\cS_x^k$ with
$k\geq i$, so its image must lie in
$\bigcap_{k\geq i}\cS_y^k=\cS_y^i$.  Thus the blocks with $j\leq i$ are allowed,
whereas those with $i<j$ are precisely the forbidden graded blocks represented by
$\cE_a$.

This description shows locally that $\rho_a$ has constant rank, hence that
$\cP_a$ and $\cE_a$ are locally free.  To globalize the block description, use the
principal bundle of frames adapted to the universal flags.  Its structure group is the
product of the relevant parabolic subgroups. The allowed matrix subspace is stable under these parabolics, as is the
corresponding quotient representation.  Their associated graded representations
are the indicated sums of block homomorphisms. Passing to associated
bundles gives \eqref{eq:gr-P} and \eqref{eq:gr-E}.  The rank formulas follow.
\end{proof}

For a two-step flag the block pattern is triangular.  Rows correspond to target
graded pieces and columns to source graded pieces; $A$ denotes an allowed block and
$F$ a forbidden one.
\begin{figure}[ht]
\centering
\renewcommand{\arraystretch}{1.15}
\begin{tabular}{c|ccc}
 & $\cG_x^1$ & $\cG_x^2$ & $\cG_x^3$\\ \hline
$\cG_y^1$ & \colorbox{black!8}{$\ A\ $} & \colorbox{black!8}{$\ A\ $} & \colorbox{black!8}{$\ A\ $}\\
$\cG_y^2$ & \colorbox{black!25}{$\ F\ $} & \colorbox{black!8}{$\ A\ $} & \colorbox{black!8}{$\ A\ $}\\
$\cG_y^3$ & \colorbox{black!25}{$\ F\ $} & \colorbox{black!25}{$\ F\ $} & \colorbox{black!8}{$\ A\ $}
\end{tabular}
\caption{Allowed and forbidden blocks for a two-step filtration.}
\label{fig:forbidden-blocks}
\end{figure}

\begin{remark}\label{rem:not-levelwise-direct-sum}
The target in \eqref{eq:rho-a} generally has larger rank than $\cE_a$.  For example,
for a two-step flag the block
$\cG_x^1\to\cG_y^3$ appears in the equation at the first step and again in the
equation at the second step.  Definition~\ref{def:P-E} counts this block once, as it
must.  Any formula using the full direct sum in \eqref{eq:rho-a} therefore has the
wrong rank and the wrong expected dimension.
\end{remark}

Let
\[ R=\Rep(Q,\alpha)=\bigoplus_{a\in Q_1} \Hom\bigl(\kk^{\alpha(t(a))},\kk^{\alpha(h(a))}\bigr). \]
The quotient maps $\cH_a\twoheadrightarrow\cE_a$ assemble to a surjective bundle
map
\begin{equation}\label{eq:universal-evaluation}
 \Phi:R\otimes\cO_X\longrightarrow\cE_{\bbeta}.
\end{equation}
For $V\in R$, denote by $s_V$ the section obtained by evaluating \eqref{eq:universal-evaluation}
on the constant vector $V$.

\begin{proposition}\label{prop:incidence-vector-bundle}
Let
\[ Z_{\bbeta}= \left\{(V,F)\in R\times X\ \middle|\ F\text{ is a flag of subrepresentations of }V\right\}. \]
Then:
\begin{enumerate}
 \item For every $V\in R$, one has a scheme-theoretic equality
 \[ \Fl_{\bbeta}(V)=Z(s_V)\subset X. \]
 \item The projection $p:Z_{\bbeta}\to X$ is the total space of the vector bundle
 $\cP_{\bbeta}=\ker\Phi$.
 \item Consequently, $Z_{\bbeta}$ is smooth and irreducible, and
 \[ \dim Z_{\bbeta}=\dim R+\dim X-\rk\cE_{\bbeta}. \]
\end{enumerate}
\end{proposition}

\begin{proof}
At a flag $F$, the value of the $a$th component of $s_V$ is the class of the arrow map
$V(a)$ modulo the filtration-preserving subspace $\cP_a|_F$.  It vanishes precisely
when
\[ V(a)(F_x^k)\subseteq F_y^k \qquad(1\leq k\leq m), \]
for every arrow $a:x\to y$.  On a splitting chart for the two flags, these are the
vanishing of the forbidden block coordinates from
Lemma~\ref{lem:block-filtration}; the incidence equations at the individual flag
levels generate the same ideal.  Thus the equality is scheme-theoretic.

The fiber of $p$ over $F$ is the vector space of representations whose arrow maps
preserve $F$, namely $\cP_{\bbeta}|_F$. 
This is Wolf's incidence-bundle construction \cite[Lemmas~4.1--4.2]{Wolf} in the present notations.
Since
\eqref{eq:universal-evaluation} is a surjection of vector bundles, this identifies
$Z_{\bbeta}$ with $\Tot(\cP_{\bbeta})$.  The remaining assertions follow because
$X$ is smooth and irreducible.
\end{proof}

\begin{proposition}\label{prop:expected-dimension}
Set
\[ r_{\bbeta}=\rk\cE_{\bbeta}, \qquad d_{\bbeta}=\dim X-r_{\bbeta}. \]
Then
\begin{align}
 \dim X
 &=\sum_{x\in Q_0}\sum_{1\leq i<j\leq m+1}
 \delta^i(x)\delta^j(x),
 \label{eq:dim-X}\\
 r_{\bbeta}
 &=\sum_{a\in Q_1}\sum_{1\leq i<j\leq m+1}
 \delta^i(t(a))\delta^j(h(a)),\notag
 \\
 d_{\bbeta}
 &=\sum_{1\leq i<j\leq m+1}
 \langle\delta^i,\delta^j\rangle_Q
 =\sum_{k=1}^{m}
 \left\langle\beta^k,\beta^{k+1}-\beta^k\right\rangle_Q.\notag
\end{align}
In particular,
\[ \dim Z_{\bbeta}-\dim R=d_{\bbeta}. \]
\end{proposition}

\begin{proof}
For each vertex $x$, the standard dimension formula for the indicated partial
flag variety gives \eqref{eq:dim-X}.  The rank formula is
Lemma~\ref{lem:block-filtration}.  Their difference gives the first expression for
$d_{\bbeta}$.  Finally,
\[
 \sum_{k=1}^{m}
 \left\langle\beta^k,\beta^{k+1}-\beta^k\right\rangle_Q
 =\sum_{k=1}^{m}\sum_{i\leq k}
 \langle\delta^i,\delta^{k+1}\rangle_Q
 =\sum_{i<j}\langle\delta^i,\delta^j\rangle_Q.
\]
The last assertion follows from Proposition~\ref{prop:incidence-vector-bundle}.
\end{proof}

\section{Generic fibers, cycle classes, and the filtered \texorpdfstring{$\Ext$}{Ext} criterion}

Let $q:Z_{\bbeta}\to R$ be the first projection, and set
$A_{\bbeta}=q(Z_{\bbeta})$.  Since $q$ is projective, $A_{\bbeta}$ is a closed
irreducible subvariety of $R$.  We write
\[ \bbeta\hookrightarrow\alpha \quad\Longleftrightarrow\quad q\text{ is dominant}. \]
Thus $\bbeta\hookrightarrow\alpha$ means $A_{\bbeta}=R$, or equivalently that
a general $\alpha$-dimensional representation has a flag of type $\bbeta$.

\begin{theorem}[Flag varieties of general representations as regular zero loci]\label{thm:generic-regular-zero-locus}
Let $\kk$ be algebraically closed.
\begin{enumerate}
 \item If $q$ is not dominant, then $\Fl_{\bbeta}(V)=\varnothing$ for general
 $V\in R$, and
 \[ c_{r_{\bbeta}}(\cE_{\bbeta})=0\quad\text{in }A^{r_{\bbeta}}(X). \]
 \item If $q$ is dominant, then $d_{\bbeta}\geq0$, and there is a nonempty open
 subset $R^\circ\subseteq R$ such that for every $V\in R^\circ$ the section $s_V$
 is regular, its zero scheme has pure dimension $d_{\bbeta}$, and
 \begin{equation}\label{eq:flag-fundamental-class}
  (i_V)_*[\Fl_{\bbeta}(V)]
  =c_{r_{\bbeta}}(\cE_{\bbeta})\cap[X]
  \quad\text{in }A_*(X).
 \end{equation}
 \item If, in addition, $\operatorname{char}\kk=0$, then $R^\circ$ may be chosen so
 that every $\Fl_{\bbeta}(V)$ with $V\in R^\circ$ is smooth and projective of
 dimension $d_{\bbeta}$.
\end{enumerate}
\end{theorem}

\begin{proof}
If $q$ is not dominant, choose $V$ outside the closed image $A_{\bbeta}$.  Then
$s_V$ is nowhere zero.  A nowhere-vanishing section gives a line subbundle
$\cO_X\hookrightarrow\cE_{\bbeta}$ with locally free quotient, so the top Chern class
of $\cE_{\bbeta}$ vanishes.

Suppose that $q$ is dominant.  By the fiber-dimension theorem and
Proposition~\ref{prop:expected-dimension}, there is a nonempty open subset of $R$ on
which the fibers have dimension
\[ \dim Z_{\bbeta}-\dim R=d_{\bbeta}. \]
In particular $d_{\bbeta}\geq0$.  Fix a fiber over this open set.  Locally on the
smooth variety $X$, its ideal is generated by the $r_{\bbeta}$ components of $s_V$.
Every irreducible component of a zero scheme of $r_{\bbeta}$ equations has codimension
at most $r_{\bbeta}$, hence dimension at least $\dim X-r_{\bbeta}=d_{\bbeta}$.  Since
the whole fiber has dimension $d_{\bbeta}$, every component has codimension exactly
$r_{\bbeta}$.  At every point, the localized ideal therefore has height
$r_{\bbeta}$.  A regular local ring is Cohen--Macaulay, and an ideal generated by
$r_{\bbeta}$ elements of height $r_{\bbeta}$ is generated by a regular sequence.
Thus $s_V$ is a regular section.  The standard zero-section formula gives
\eqref{eq:flag-fundamental-class}; see \cite[Example~14.1.1]{Fulton}.

Assume now that $\operatorname{char}\kk=0$.  The incidence variety is smooth by
Proposition~\ref{prop:incidence-vector-bundle}, while $R$ is an affine space.  Generic
smoothness gives a nonempty open subset over which the dominant morphism $q$ is
smooth.  Intersecting it with the open set above proves the final assertion.  Projectivity
of the fibers follows from projectivity of $q$.
\end{proof}

\begin{corollary}[The zero-dimensional case]\label{cor:zero-dimensional-flags}
Assume $\operatorname{char}\kk=0$, $q$ is dominant, and $d_{\bbeta}=0$.  For general
$V\in R$, the scheme $\Fl_{\bbeta}(V)$ is finite and reduced, and
\[ \#\Fl_{\bbeta}(V) =\int_X c_{\dim X}(\cE_{\bbeta}). \]
Thus the top-Chern degree is the enumerative quiver-flag analogue of the
Crawley--Boevey count for quiver Grassmannians \cite{CB}.
\end{corollary}

\begin{proof}
Theorem~\ref{thm:generic-regular-zero-locus} makes the general fiber smooth,
projective, and zero-dimensional, hence finite and reduced.  Its length is the degree
of the class in \eqref{eq:flag-fundamental-class}.
\end{proof}

\begin{corollary}[Top-Chern-class criterion for generic existence]\label{cor:chern-dominance}
One has
\[ \bbeta\hookrightarrow\alpha \quad\Longleftrightarrow\quad c_{r_{\bbeta}}(\cE_{\bbeta})\neq0 \quad\text{in }A^{r_{\bbeta}}(X). \]
Consequently, generic existence of a flag of type $\bbeta$ is independent of the
algebraically closed ground field.
\end{corollary}

\begin{proof}
Theorem~\ref{thm:generic-regular-zero-locus} proves the implication from
non-dominance to vanishing. If $q$
is dominant, choose a general fiber $Y=\Fl_{\bbeta}(V)$.  It is nonempty and
\eqref{eq:flag-fundamental-class} identifies the top Chern class with the class of the
effective cycle $Y$.  If $H$ is ample on $X$, then
$\deg(H^{d_{\bbeta}}\cap[Y])>0$, so $[Y]$ cannot vanish in the Chow group.

For the final statement, the split partial flag variety and its universal bundles are
defined over $\ZZ$.  Schubert classes give an integral basis of its Chow ring over any
field, and the coefficients of $c_{r_{\bbeta}}(\cE_{\bbeta})$ in that basis are the same
integers after every base change.  Its nonvanishing, and hence dominance, is therefore
field-independent.
\end{proof}

\subsection{The ladder algebra and the generic filtered \texorpdfstring{$\Ext^1$}{Ext1}-dimension}

The tangent-space calculation below is valid in arbitrary characteristic.  Let
$A_{m+1}$ be the linearly oriented quiver
\[ 1\longrightarrow2\longrightarrow\cdots\longrightarrow m+1 \]
and set
\[ \Lambda=\kk Q\otimes_\kk\kk A_{m+1}. \]
Equivalently, $\Lambda$ is the path algebra of the product quiver
$Q\times A_{m+1}$ modulo the commuting-square relations.  A finite-dimensional
$\Lambda$-module is a chain of $Q$-representations and $Q$-linear maps.

For $V\in R$, let
\[ \bV=(V\xrightarrow{1}V\xrightarrow{1}\cdots\xrightarrow{1}V) \]
be the constant $\Lambda$-module.  A flag
$U=(U^1\subset\cdots\subset U^m\subset V)$ determines a submodule
\[ \bU=(U^1\hookrightarrow U^2\hookrightarrow\cdots\hookrightarrow U^m \hookrightarrow V) \subseteq\bV, \]
and its quotient is
\[ \bW=\bV/\bU =(V/U^1\longrightarrow V/U^2\longrightarrow\cdots\longrightarrow V/U^m \longrightarrow0). \]

\begin{proposition}\label{prop:tangent-ext}
At a point $U\in\Fl_{\bbeta}(V)$ there are canonical identities
\begin{align}
 T_U\Fl_{\bbeta}(V)
 &\cong\Hom_\Lambda(\bU,\bW),
 \label{eq:tangent-hom-lambda}\\
 \dim T_U\Fl_{\bbeta}(V)
 &=d_{\bbeta}+\dim\Ext_\Lambda^1(\bU,\bW).
 \label{eq:tangent-d-plus-ext}
\end{align}
Moreover,
\[ \Ext_\Lambda^2(\bU,\bW)=0. \]
\end{proposition}

\begin{proof}
The quiver flag scheme is the $\Lambda$-module Grassmannian of the
constant module $\bV$ with dimension vectors
$\beta^1,\ldots,\beta^m,\alpha$ at the successive ladder vertices.  The standard
dual-number calculation for a module Grassmannian identifies its Zariski tangent space
at $\bU$ with $\Hom_\Lambda(\bU,\bV/\bU)$, giving
\eqref{eq:tangent-hom-lambda}; see also \cite[Lemma~4.1]{Wolf}.

The Euler form of $\Lambda$ is
\[ \langle\bM,\bN\rangle_\Lambda =\sum_{k=1}^{m+1}\langle\dim M^k,\dim N^k\rangle_Q -\sum_{k=1}^{m}\langle\dim M^k,\dim N^{k+1}\rangle_Q. \]
For $\bU$ and $\bW$, this becomes
\begin{align*}
 \langle\bU,\bW\rangle_\Lambda
 &=\sum_{k=1}^{m}
 \left(
 \langle\beta^k,\alpha-\beta^k\rangle_Q
 -\langle\beta^k,\alpha-\beta^{k+1}\rangle_Q
 \right)\\
 &=\sum_{k=1}^{m}
 \langle\beta^k,\beta^{k+1}-\beta^k\rangle_Q
 =d_{\bbeta}.
\end{align*}

The path algebra $\kk Q$ is hereditary even when $Q$ has oriented cycles: it is the
tensor algebra of the arrow bimodule over the semisimple algebra $\kk^{Q_0}$, although
it may then be infinite-dimensional.  The standard tensor-ring resolution for the
ladder algebra gives
$\operatorname{gldim}(\Lambda)\leq2$; in the present notation this is the resolution
constructed in \cite[Theorem~3.2]{Wolf}.  If $P_1$ denotes the indecomposable
projective $\kk A_{m+1}$-module at the first vertex, then
$\bV\cong V\otimes_\kk P_1$.  Tensoring a length-one projective resolution of $V$
with $P_1$ shows that $\bV$ has projective dimension at most one.  From
$0\to\bU\to\bV\to\bW\to0$ and the global-dimension bound, one obtains
$\Ext_\Lambda^2(\bU,N)=0$ for every $N$, hence in particular
$\Ext_\Lambda^2(\bU,\bW)=0$.  The Euler identity
\[ \langle\bU,\bW\rangle_\Lambda =\dim\Hom_\Lambda(\bU,\bW) -\dim\Ext_\Lambda^1(\bU,\bW) +\dim\Ext_\Lambda^2(\bU,\bW) \]
now gives \eqref{eq:tangent-d-plus-ext}.
\end{proof}

\begin{definition}
Define the \emph{generic filtered $\Ext^1$-dimension}
\[ \fext_Q(\bbeta;\alpha) =\min_{(V,U)\in Z_{\bbeta}} \dim_\kk\Ext_\Lambda^1(\bU,\bV/\bU). \]
\end{definition}

\begin{theorem}[Filtered \texorpdfstring{$\Ext$}{Ext} criterion]\label{thm:filtered-ext-criterion}
Let $A_{\bbeta}=q(Z_{\bbeta})$.  In arbitrary characteristic,
\begin{equation}\label{eq:codim-fext-inequality}
 \codim_R A_{\bbeta}\leq\fext_Q(\bbeta;\alpha).
\end{equation}
If $\operatorname{char}\kk=0$, then equality holds:
\begin{equation}\label{eq:codim-fext}
 \codim_R A_{\bbeta}=\fext_Q(\bbeta;\alpha).
\end{equation}
Consequently, in characteristic zero,
\[ \bbeta\hookrightarrow\alpha \quad\Longleftrightarrow\quad \fext_Q(\bbeta;\alpha)=0. \]
In positive characteristic, the implication
$\fext_Q(\bbeta;\alpha)=0\Rightarrow\bbeta\hookrightarrow\alpha$ still holds.
\end{theorem}

\begin{proof}
Put $c=\codim_R A_{\bbeta}$.  Regard $q$ as a dominant morphism
$Z_{\bbeta}\to A_{\bbeta}$.  For every $(V,U)\in Z_{\bbeta}$, the local dimension of the fiber is at least
\[ \dim Z_{\bbeta}-\dim A_{\bbeta}=d_{\bbeta}+c. \]
It is at most the dimension of its Zariski tangent space.  By
Proposition~\ref{prop:tangent-ext},
\[ d_{\bbeta}+c \leq\dim_U\Fl_{\bbeta}(V) \leq\dim T_U\Fl_{\bbeta}(V) =d_{\bbeta}+\dim\Ext_\Lambda^1(\bU,\bW). \]
Thus $c\leq\dim\Ext_\Lambda^1(\bU,\bW)$ for every point of the incidence
variety, proving \eqref{eq:codim-fext-inequality}.

In characteristic zero, generic smoothness applied over the smooth locus of
$A_{\bbeta}$ gives a point at which the fiber is smooth of dimension
$d_{\bbeta}+c$.  At that point the tangent-space dimension equals the fiber dimension,
so Proposition~\ref{prop:tangent-ext} gives
$\dim\Ext_\Lambda^1(\bU,\bW)=c$.  This proves \eqref{eq:codim-fext}. The dominance criterion in characteristic zero follows immediately; the
positive-characteristic implication follows from
\eqref{eq:codim-fext-inequality}. This argument is the ladder-algebra form of
\cite[Theorem~4.5 and Remark~4.6]{Wolf}.
\end{proof}

\begin{remark}\label{rem:one-step-ext}
When $m=1$, write $\beta=\beta^1$ and $\gamma=\alpha-\beta$.  Then
$\Hom_\Lambda(\bU,\bW)=\Hom_Q(U,V/U)$ and the two Euler forms agree.
Since both relevant second extension groups vanish, it follows that
\[ \dim\Ext_\Lambda^1(\bU,\bW)=\dim\Ext_Q^1(U,V/U). \]
Conversely, any pair $(M,N)$ of dimensions $(\beta,\gamma)$ occurs from the split
extension $M\subset M\oplus N$.  Hence
\[ \fext_Q(\beta;\alpha)=\ext_Q(\beta,\gamma). \]
When $Q$ is acyclic, the dominance statement in
Theorem~\ref{thm:filtered-ext-criterion} is Schofield's criterion
\cite[Theorem~3.3]{SchofieldGeneral}.
\end{remark}

\section{Characteristic classes of the generic flag variety}

From now on we work over $\CC$ and fix an arbitrary flag type,
without assuming dominance.  Put $X=X_{\bbeta}$, $d=d_{\bbeta}$, and
$r=r_{\bbeta}$.

\begin{lemma}\label{lem:tangent-filtration}
The tangent bundle of $X$ has associated graded
\begin{equation}\label{eq:gr-TX}
 \gr T_X\cong
 \bigoplus_{x\in Q_0}\bigoplus_{i<j}
 (\cG_x^i)^\vee\otimes\cG_x^j.
\end{equation}
Consequently,
\begin{align}
 c(T_X)
 &=\prod_{x\in Q_0}\prod_{i<j}
 c\bigl((\cG_x^i)^\vee\otimes\cG_x^j\bigr),\notag
 \\
 c(\cE_{\bbeta})
 &=\prod_{a:x\to y}\prod_{i<j}
 c\bigl((\cG_x^i)^\vee\otimes\cG_y^j\bigr).
 \label{eq:c-E-product}
\end{align}
\end{lemma}

\begin{proof}
For $X_x=\GL_{\alpha(x)}/P_x$, filter the $P_x$-module
$\mathfrak{gl}_{\alpha(x)}/\mathfrak p_x$ by block distance.  Its associated graded
is $\bigoplus_{i<j}\Hom(\kk^{\delta^i(x)},\kk^{\delta^j(x)})$.  Passing to
associated bundles and taking the product over $x$ gives \eqref{eq:gr-TX}. The two Chern-class identities follow from multiplicativity of Chern classes under filtrations and from Lemma~\ref{lem:block-filtration}.
\end{proof}

\begin{corollary}\label{cor:explicit-top-chern}
For every flag type, the top Chern class of the equation bundle is
\[ c_r(\cE_{\bbeta}) =\prod_{a:x\to y}\prod_{i<j} c_{\delta^i(x)\delta^j(y)} \bigl((\cG_x^i)^\vee\otimes\cG_y^j\bigr). \]
If $u_{x,i,1},\ldots,u_{x,i,\delta^i(x)}$ are formal Chern roots of $\cG_x^i$, then
\begin{equation}\label{eq:top-chern-roots}
 c_r(\cE_{\bbeta})
 =\prod_{a:x\to y}\prod_{i<j}
 \prod_{p=1}^{\delta^i(x)}\prod_{q=1}^{\delta^j(y)}
 \bigl(u_{y,j,q}-u_{x,i,p}\bigr).
\end{equation}
If the incidence projection is dominant and $Y=\Fl_{\bbeta}(V)$ is a general
smooth fiber, then $i_*[Y]=c_r(\cE_{\bbeta})\cap[X]$, where $i:Y\hookrightarrow X$.
\end{corollary}

\begin{proof}
The Chern-class identities follow from \eqref{eq:c-E-product}. The splitting principle gives the
root formula. In the dominant case the identification with the fundamental class is
Theorem~\ref{thm:generic-regular-zero-locus}.
\end{proof}

Assume now that the incidence projection is dominant, and let
$Y=\Fl_{\bbeta}(V)$ be a general smooth fiber with inclusion
$i:Y\hookrightarrow X$.  Regularity identifies its normal bundle with
$\cE_{\bbeta}|_Y$, hence
\begin{equation}\label{eq:normal-exact-sequence}
 0\longrightarrow T_Y\longrightarrow T_X|_Y
 \longrightarrow\cE_{\bbeta}|_Y\longrightarrow0,
 \qquad
 c(T_Y)=i^*\!\left(\frac{c(T_X)}{c(\cE_{\bbeta})}\right).
\end{equation}

\begin{theorem}[Characteristic-class formula]\label{thm:flag-euler-integral}
For a general $V$ in the dominant case, we have
\begin{equation}\label{eq:flag-euler-integral}
 \topchi(Y)
 =\int_X c_r(\cE_{\bbeta})
 \left[\frac{c(T_X)}{c(\cE_{\bbeta})}\right]_d.
\end{equation}
More generally, if $P(c_1,\ldots,c_d)$ is homogeneous of weighted degree $d$, then
\[ \int_Y P\bigl(c_1(T_Y),\ldots,c_d(T_Y)\bigr) =\int_X c_r(\cE_{\bbeta}) P\bigl(c_1(T_X-\cE_{\bbeta}),\ldots,c_d(T_X-\cE_{\bbeta})\bigr). \]
If $\Phi$ is a multiplicative characteristic class in the sense of
Hirzebruch \cite{Hi}, then
\[ \int_Y[\Phi(T_Y)]_d =\int_X c_r(\cE_{\bbeta}) \left[\frac{\Phi(T_X)}{\Phi(\cE_{\bbeta})}\right]_d. \]
\end{theorem}

\begin{proof}
Apply \eqref{eq:normal-exact-sequence}, the projection formula, and
$i_*[Y]=c_r(\cE_{\bbeta})\cap[X]$.  Formula \eqref{eq:flag-euler-integral} is the
case $P=c_d$ together with the Chern--Gauss--Bonnet theorem \cite{Ch}; the
multiplicative-class formula follows in the same way.
\end{proof}

Using Lemma~\ref{lem:tangent-filtration}, the Euler integrand may be written
as
\begin{align*}
 \prod_{a:x\to y}\prod_{i<j}
 c_{\delta^i(x)\delta^j(y)}
 \bigl((\cG_x^i)^\vee\otimes\cG_y^j\bigr)
 \left[
 \frac{\displaystyle\prod_{x\in Q_0}\prod_{i<j}
 c\bigl((\cG_x^i)^\vee\otimes\cG_x^j\bigr)}
 {\displaystyle\prod_{a:x\to y}\prod_{i<j}
 c\bigl((\cG_x^i)^\vee\otimes\cG_y^j\bigr)}
 \right]_d.
\end{align*}

\section{Localization}

We let
\[ \mathbb T=\prod_{x\in Q_0}(\CC^\times)^{\alpha(x)} \]
act diagonally, with algebraically independent characters
$\lambda_{x,1},\ldots,\lambda_{x,\alpha(x)}$.  The calculation below takes place in
the localized equivariant Chow ring obtained by inverting the nonzero tangent
characters.  A fixed point of $X$ is indexed by
ordered set partitions
\begin{equation}\label{eq:ordered-partition}
 [\alpha(x)]=I_x^1\sqcup\cdots\sqcup I_x^{m+1},
 \qquad |I_x^i|=\delta^i(x).
\end{equation}
For the collection $\bI=(I_x^i)_{x,i}$, define the tangent and equation-weight
multisets
\begin{align}
 \mathsf T_{\bI}
 &=\bigsqcup_{x\in Q_0}\bigsqcup_{i<j}
 \{\lambda_{x,q}-\lambda_{x,p}:p\in I_x^i,\ q\in I_x^j\},
 \label{eq:flag-tangent-weights}\\
 \mathsf E_{\bI}
 &=\bigsqcup_{a:x\to y}\bigsqcup_{i<j}
 \{\lambda_{y,q}-\lambda_{x,p}:p\in I_x^i,\ q\in I_y^j\}.
 \label{eq:flag-equation-weights}
\end{align}
Arrow and weight multiplicities are counted.  Set
\[ \mathsf H_d(\bI) =[z^d]\, \frac{\prod_{\tau\in\mathsf T_{\bI}}(1+z\tau)} {\prod_{\varepsilon\in\mathsf E_{\bI}}(1+z\varepsilon)}. \]

\begin{theorem}[Localization on the flag product]\label{thm:flag-localization}
Let $\Xi_d\in A^d(X)$ be any homogeneous class, with an equivariant lift
$\widetilde\Xi_d\in A^d_{\mathbb T}(X)$.  For every flag type, localization on $X$ gives
\begin{equation}\label{eq:top-chern-localization}
 \int_X c_r(\cE_{\bbeta})\,\Xi_d
 =\sum_{\bI}
 \frac{\prod_{\varepsilon\in\mathsf E_{\bI}}\varepsilon}
 {\prod_{\tau\in\mathsf T_{\bI}}\tau}\,
 \widetilde\Xi_d|_{p_{\bI}}.
\end{equation}
Here the sum runs over \eqref{eq:ordered-partition}; if $d<0$, both sides
are interpreted as zero.

If the incidence projection is dominant, then for general $V$ one obtains
\begin{equation}\label{eq:flag-euler-localization}
 \topchi\bigl(\Fl_{\bbeta}(V)\bigr)
 =\sum_{\bI}
 \frac{\prod_{\varepsilon\in\mathsf E_{\bI}}\varepsilon}
 {\prod_{\tau\in\mathsf T_{\bI}}\tau}\,
 \mathsf H_d(\bI).
\end{equation}
The rational expressions are independent of the chosen characters and are
integers in the Euler-characteristic case.
\end{theorem}

\begin{proof}
At $p_{\bI}$, Lemmas~\ref{lem:tangent-filtration} and
\ref{lem:block-filtration} give the weight multisets
\eqref{eq:flag-tangent-weights} and \eqref{eq:flag-equation-weights}.
The equivariant top Chern class restricts to $\prod\varepsilon$, and the
equivariant Euler class of $T_X$ to $\prod\tau$.  Applying equivariant
localization \cite{EdidinGraham} to these classes on $X$ gives
\eqref{eq:top-chern-localization}.  A general section $s_V$ need not be
$\mathbb T$-equivariant; no torus action on its zero locus is used.
For $\Xi_d=[c(T_X)/c(\cE_{\bbeta})]_d$, choose the corresponding
equivariant characteristic-class lift.  Its restriction to $p_{\bI}$ is
$\mathsf H_d(\bI)$, so in the dominant case
Theorem~\ref{thm:flag-euler-integral} yields
\eqref{eq:flag-euler-localization}.  These are ordinary degree-zero
integrals; hence the sums do not depend on the character values.  The same
localization identity computes every characteristic-number pairing.
\end{proof}

\begin{remark}[Implementation]\label{rem:flag-computation}
Theorem~\ref{thm:flag-localization} applies to any finite quiver.  Choose
pairwise distinct integer characters, enumerate the torus-fixed points, and
read off the tangent and bundle weights from $(Q,\alpha,\bbeta)$.  The sum
can then be evaluated in exact rational arithmetic. No Schubert
expansion or covariant calculation is needed, although the number of fixed
points can be large.  When $d=0$, we have $\mathsf H_0(\bI)=1$,
so localization computes $\deg c_r(\cE_{\bbeta})$.
\end{remark}

We apply \eqref{eq:top-chern-localization} to the triple-flag quiver
with three inward arms:
\[
\xymatrix@C=14pt@R=10pt{
 x_1 \ar[r] & \cdots \ar[r] & x_{n-1} \ar[dr] & \\
 y_1 \ar[r] & \cdots \ar[r] & y_{n-1} \ar[r] & c \\
 z_1 \ar[r] & \cdots \ar[r] & z_{n-1} \ar[ur] &
}
\]
The vertices $x_k,y_k,z_k$ have dimension $k$, and $c$ has dimension $n$.
Fix $r\leq n$.  For a partition $\theta$ in the $r\times(n-r)$ rectangle, put
\[ p_i(\theta)=n-r+i-\theta_i\quad(1\leq i\leq r),\qquad
   \beta_\theta(k)=\#\{i:p_i(\theta)\leq k\}\quad(1\leq k<n). \]
Thus $p_i(\theta)$ are the Schubert positions: relative to a complete flag
$F_\bullet$, the conditions $\dim(U\cap F_{p_i(\theta)})\geq i$
define the Schubert variety of codimension $|\theta|$ in $\Gr(r,n)$.
For three such partitions $\lambda,\mu,\nu$, take the one-step subdimension
vector with value $r$ at $c$ and values
$\beta_\lambda(k),\beta_\mu(k),\beta_\nu(k)$ on the corresponding arms.
The multisets $\mathsf T_{\bI}$ and $\mathsf E_{\bI}$ below are those of
\eqref{eq:flag-tangent-weights} and \eqref{eq:flag-equation-weights}
for this quiver and flag type.

\begin{corollary}\label{cor:LR-localization}
Suppose that $|\lambda|+|\mu|+|\nu|=r(n-r)$, and set
$\nu^\vee=(n-r-\nu_r,\ldots,n-r-\nu_1)$.  Then
\begin{equation}\label{eq:LR-localization}
 c_{\lambda,\mu}^{\nu^\vee}
 =\sum_{\bI}
 \frac{\prod_{\varepsilon\in\mathsf E_{\bI}}\varepsilon}
 {\prod_{\tau\in\mathsf T_{\bI}}\tau}.
\end{equation}
\end{corollary}

\begin{proof}
Along each arm, the incidence conditions are the standard resolution of the
Schubert condition of shape $\theta$ on the central $r$-plane, and proper
pushforward to $\Gr(r,n)$ gives the Schubert class $\sigma_\theta$
\cite{DW,DSW}.  Moreover, the expected dimension of the quiver-Grassmannian
is
\[ d=r(n-r)-|\lambda|-|\mu|-|\nu|=0. \]
For general complete flags, the triple Schubert intersection is transverse
and avoids the boundaries of the open Schubert cells.  Each arm resolution
is an isomorphism over its open cell, so the incidence fiber has the same
degree as the intersection.  If the intersection is empty, a general section
has no zero, and its top Chern class vanishes.
Thus the top Chern number of the equation bundle is the degree of
the triple Schubert intersection:
\[ \int_Xc_{\dim X}(\cE_{\bbeta}) =\int_{\Gr(r,n)}\sigma_\lambda\sigma_\mu\sigma_\nu =c_{\lambda,\mu}^{\nu^\vee}. \]
Apply \eqref{eq:top-chern-localization} with $\Xi_0=1$.
The localization identity is an integral on $X$ and therefore also applies
when $c_{\lambda,\mu}^{\nu^\vee}=0$ and the incidence projection is not dominant.
\end{proof}

\begin{remark}\label{rem:LR-other-formulas}
Formula \eqref{eq:LR-localization} is a finite fixed-point expression indexed by
coordinate subspaces.  Wen obtains a closed formula for the same invariant spaces from
parabolic theta functions and the Verlinde formula \cite{Wen}, while Rietsch obtains a
different roots-of-unity expression from quantum cohomology of Grassmannians
\cite{Rietsch}.
\end{remark}

\section{The parabolic Schubert expansion}

We use the conventions $A^j(X)=0$ and $[\Phi]_j=0$ for $j<0$, and a sum indexed by
$\ell(\bw)=j$ is empty when $j<0$.  Thus the formulas below also cover negative
expected dimension without separate qualifications.

Neither acyclicity nor dominance is assumed in this section.  If the
incidence projection is not dominant, then
$c_{r_{\bbeta}}(\cE_{\bbeta})=0$ by Corollary~\ref{cor:chern-dominance}, so the
Schubert coefficients below vanish.  The characteristic numbers of a general
nonempty fiber are interpreted under dominance.

For a vertex $x$, put
\[ G_x=\GL_{\alpha(x)}, \qquad P_x\subseteq G_x \]
for the parabolic subgroup stabilizing the chosen flag type at $x$.  Let
$W_x\cong\mathfrak S_{\alpha(x)}$ be the Weyl group, and let $W_x^{P_x}$ denote
the minimal-length representatives for $W_x/W_{P_x}$.  Write $\sigma_{w_x}$ for the opposite Schubert class of codimension
$\ell(w_x)$ on $X_x=G_x/P_x$.  For
\[ \bw=(w_x)_{x\in Q_0}\in\prod_{x\in Q_0}W_x^{P_x}, \]
set
\[ \ell(\bw)=\sum_x\ell(w_x), \qquad \sigma_{\bw}=\prod_x\sigma_{w_x}. \]
Let $\sigma_{\bw}^{\vee}$ denote the Poincar\'e dual Schubert class, characterized by
\begin{equation}\label{eq:parabolic-duality}
 \int_X \sigma_{\bw}^{\vee}\sigma_{\bv}=\delta_{\bw,\bv}
 \qquad\text{whenever }\ell(\bw)=\ell(\bv).
\end{equation}

Since $c_r(\cE_{\bbeta})$ has codimension $r=\dim X-d$, there is a unique expansion
\begin{equation}\label{eq:flag-cycle-schubert}
 c_r(\cE_{\bbeta})
 =\sum_{\ell(\bw)=d}N_{\bw}(Q,\alpha,\bbeta)\,\sigma_{\bw}^{\vee}.
\end{equation}
Equivalently,
\begin{equation}\label{eq:Nw-intersection}
 N_{\bw}(Q,\alpha,\bbeta)
 =\int_X c_r(\cE_{\bbeta})\sigma_{\bw}.
\end{equation}
\begin{definition}\label{def:generic-quiver-flag-schubert-coefficient}
The integers $N_{\bw}(Q,\alpha,\bbeta)$ are the \emph{Schubert coefficients of
the generic quiver-flag cycle}.  They are defined for every flag type; in the
non-dominant case they all vanish.
\end{definition}

\begin{remark}\label{rem:BVW-schubert-positions}
In the one-step case, Baldoni, Vergne, and Walter give inductive Horn-type conditions
for the Schubert positions occupied by subrepresentations of a general quiver
representation \cite{BVW}.  The coefficients $N_{\bw}$ refine this existence problem in
complementary degree: they are intersection degrees of the generic
subrepresentation cycle with Schubert conditions.  The present construction also
allows simultaneous multi-step conditions.
\end{remark}

We define the degree-$d$ virtual tangent class
\begin{equation}\label{eq:flag-Theta}
 \Theta_d(Q,\alpha,\bbeta)
 =\left[\frac{c(T_X)}{c(\cE_{\bbeta})}\right]_d
 =\sum_{\ell(\bw)=d}a_{\bw}(Q,\alpha,\bbeta)\,\sigma_{\bw}.
\end{equation}

\begin{proposition}\label{prop:geometric-pairing}
The coefficients $N_{\bw}$ are nonnegative integers, and the coefficients $a_{\bw}$ are
integers. For general $V\in\Rep(Q,\alpha)$, we have
\begin{equation}\label{eq:flag-geometric-pairing}
 \topchi\bigl(\Fl_{\bbeta}(V)\bigr)
 =\sum_{\ell(\bw)=d}a_{\bw}(Q,\alpha,\bbeta)
 N_{\bw}(Q,\alpha,\bbeta).
\end{equation}
For such $V$, every Chern number and every multiplicative characteristic genus is a
finite integral, or rational as appropriate, linear combination of the same
coefficients $N_{\bw}$.
\end{proposition}

\begin{proof}
The classes of Schubert varieties form an integral basis of $A^*(X)$, so the coefficients
in both expansions are integral.  If the incidence projection is not dominant, then
$c_r(\cE_{\bbeta})=0$ by Corollary~\ref{cor:chern-dominance}, and hence every
$N_{\bw}$ is zero.  The general fiber is empty, so the characteristic-number integrals
also vanish and the assertions are immediate.  Assume that the incidence projection
is dominant, and let
$Y=\Fl_{\bbeta}(V)$ be a general smooth fiber.  By
\eqref{eq:flag-fundamental-class}, $N_{\bw}=\int_Y i^*\sigma_{\bw}$.  A general
translate of a Schubert variety representing $\sigma_{\bw}$ meets $Y$ properly, and
the resulting zero-dimensional intersection cycle has nonnegative degree.  Thus
$N_{\bw}\geq0$; one may use Kleiman transversality \cite{Kleiman} over $\CC$.

Combining Theorem~\ref{thm:flag-euler-integral},
\eqref{eq:flag-cycle-schubert}, \eqref{eq:flag-Theta}, and
\eqref{eq:parabolic-duality} gives \eqref{eq:flag-geometric-pairing}.  The same argument,
with the virtual tangent characteristic class replaced by the class in
Theorem~\ref{thm:flag-euler-integral}, proves
the final assertion.
\end{proof}

\begin{remark}[One-step specialization]
For $m=1$, with $\beta=\beta^1$ and $\gamma=\alpha-\beta$, one has
\[ \cE_{\bbeta}=\bigoplus_{a\in Q_1} \cS_{t(a)}^\vee\otimes\cQ_{h(a)}, \qquad d=\langle\beta,\gamma\rangle_Q. \]
Ordered set partitions become subsets, and the formulas above specialize to generic
quiver Grassmannians.  For acyclic $Q$, Derksen--Schofield--Weyman identify the
coefficients $N_{\bw}$ with covariant multiplicities \cite[Section~6]{DSW}.
\end{remark}

\section{The twisted exterior class and its Schur coefficients}
\label{sec:twisted-exterior-class}

We continue over $\CC$.  For each pair $v=(x,i)$ with $\delta^i(x)>0$, choose a
vector space $D_v=D_x^i$ of dimension
\[ d_v=\delta^i(x), \]
and write $\cG_v=\cG_x^i$.  We omit zero-rank graded pieces. We form the layered quiver $Q^{\triangle}$ with vertex set
\[ Q_0^{\triangle}=\{(x,i)\mid x\in Q_0,\ 1\leq i\leq m+1,\ d_{(x,i)}>0\} \]
and with one arrow
\[ b=(a;i,j):(x,i)\longrightarrow(y,j) \]
for every arrow $a:x\to y$ of $Q$ and every $i<j$ with
$\delta^i(x)\delta^j(y)>0$. Parallel arrows are kept.  The second coordinate
strictly increases along every arrow, so $Q^{\triangle}$ is acyclic even if $Q$
has oriented cycles.
Its dimension vector is $\mathbf d^{\triangle}=(d_v)_{v\in Q_0^{\triangle}}$,
where $d_{(x,i)}=\delta^i(x)$.
We set
\begin{equation*}
 G_{\bdelta}=\prod_{v\in Q_0^{\triangle}}\GL(D_v),
 \qquad
W^{\triangle}=\Rep(Q^{\triangle},\mathbf d^{\triangle})
 =\bigoplus_{b:u\to v}\Hom(D_u,D_v).
\end{equation*}
For $b:u\to v$ in $Q^{\triangle}$, put $p_b=d_u$ and $q_b=d_v$.

\subsection{The twisted exterior class}

Let $R(G_{\bdelta})$ be the Grothendieck ring of finite-dimensional rational
$G_{\bdelta}$-modules, and let $R_{\mathrm{pol}}(G_{\bdelta})$ be its polynomial
subring.  If $M$ is a finite-dimensional module, write
\[ \lambda_{-1}(M)=\sum_{k\geq0}(-1)^k[\Lambda^kM]. \]
Define the determinant twist and the homogeneous exterior terms by
\begin{align}
 \chi_{\mathrm{det}}
 &=\bigotimes_{b:u\to v}(\det D_v)^{p_b},\notag
 \\
 \mathfrak K^k_{Q,\bdelta}
 &=\chi_{\mathrm{det}}\otimes
 \Lambda^k\bigl((W^{\triangle})^*\bigr),\notag
 \\
 \mathfrak K_{Q,\bdelta}
 &=\sum_{k=0}^{r_{\bbeta}}(-1)^k
 [\mathfrak K^k_{Q,\bdelta}]
 =\chi_{\mathrm{det}}\otimes
 \lambda_{-1}\bigl((W^{\triangle})^*\bigr)
 \in R(G_{\bdelta}).
 \label{eq:twisted-exterior-class}
\end{align}
Here $\dim W^{\triangle}=r_{\bbeta}$. We call $\mathfrak K_{Q,\bdelta}$ the twisted exterior class of
$Q^{\triangle}$.  The determinant twist makes each homogeneous exterior term
polynomial, as verified below.

For a partition $\lambda\subset q^p$, meaning
$\ell(\lambda)\leq p$ and $\lambda_1\leq q$, let $\lambda'$ be its conjugate,
padded to length $q$, and define the rotated rectangular complement
\[ \lambda^{\dagger(p,q)} =\bigl(p-\lambda'_q,p-\lambda'_{q-1},\ldots,p-\lambda'_1\bigr). \]
It is a partition contained in the $q\times p$ rectangle.  When the rectangle is
clear, we write simply $\lambda^\dagger$.

\begin{lemma}\label{lem:finite-rectangular-cauchy}
Let $D$ and $E$ have dimensions $p$ and $q$.  In
$R_{\mathrm{pol}}(\GL(D)\times\GL(E))$,
\begin{equation}\label{eq:representation-rectangular-cauchy}
 (\det E)^p\otimes\lambda_{-1}(D\otimes E^*)
 =\sum_{\lambda\subset q^p}(-1)^{|\lambda|}
 [\Schur_\lambda D\otimes\Schur_{\lambda^{\dagger(p,q)}}E].
\end{equation}
Equivalently, for alphabets $x=(x_1,\ldots,x_p)$ and
$y=(y_1,\ldots,y_q)$,
\begin{equation}\label{eq:symmetric-rectangular-cauchy}
 \prod_{r=1}^{p}\prod_{s=1}^{q}(y_s-x_r)
 =\sum_{\lambda\subset q^p}(-1)^{|\lambda|}
 s_\lambda(x)s_{\lambda^{\dagger(p,q)}}(y).
\end{equation}
\end{lemma}

\begin{proof}
Exterior Cauchy gives
\[
 \Lambda^k(D\otimes E^*)
 \cong\bigoplus_{\substack{\lambda\vdash k\\ \lambda\subset q^p}}
 \Schur_\lambda D\otimes\Schur_{\lambda'}E^*.
\]
For $\lambda\subset q^p$, highest-weight duality gives
\[ (\det E)^p\otimes\Schur_{\lambda'}E^* \cong\Schur_{\lambda^{\dagger(p,q)}}E. \]
Taking the alternating sum proves \eqref{eq:representation-rectangular-cauchy}.
The Schur-character identity is \eqref{eq:symmetric-rectangular-cauchy}; it also
follows directly by expanding the product $\prod(y_s-x_r)$.  We use the standard
Schur-functor and Cauchy conventions of \cite[Chapter~I]{Macdonald}.
\end{proof}

\begin{lemma}\label{lem:homogeneous-polynomiality}
For every $0\leq k\leq r_{\bbeta}$, the module
$\mathfrak K^k_{Q,\bdelta}$ is a finite-dimensional polynomial
$G_{\bdelta}$-module.  Consequently,
$\mathfrak K_{Q,\bdelta}\in R_{\mathrm{pol}}(G_{\bdelta})$.
\end{lemma}

\begin{proof}
For an arrow $b:u\to v$ of $Q^{\triangle}$, put
$M_b=D_u\otimes D_v^*$.  The exterior power of a direct sum decomposes as
\[ \Lambda^k\!\left(\bigoplus_bM_b\right) \cong \bigoplus_{\sum_b k_b=k}\ \bigotimes_b\Lambda^{k_b}M_b. \]
After multiplying by $\chi_{\mathrm{det}}$, the summand indexed by
$(k_b)_b$ is the tensor product over the arrows of $Q^{\triangle}$ of
\[ (\det D_v)^{p_b}\otimes\Lambda^{k_b}(D_u\otimes D_v^*). \]
The degree-$k_b$ part of the proof of
Lemma~\ref{lem:finite-rectangular-cauchy} identifies this factor with
\[
 \bigoplus_{\substack{\lambda\subset q_b^{p_b}\\|\lambda|=k_b}}
 \Schur_\lambda D_u\otimes
 \Schur_{\lambda^{\dagger(p_b,q_b)}}D_v,
\]
which is polynomial.  Direct sums and tensor products preserve polynomiality.
\end{proof}

\begin{proposition}\label{prop:associated-bundle-realization}
Let
\[ \Phi_K:R(G_{\bdelta})\longrightarrow K^0(X) \]
be the associated-bundle homomorphism determined by $D_v\mapsto\cG_v$.  Then
\begin{equation}\label{eq:associated-bundle-class}
 \Phi_K(\mathfrak K_{Q,\bdelta})
 =\bigotimes_{b:u\to v}
 \left((\det\cG_v)^{p_b}\otimes
 \lambda_{-1}(\cG_u\otimes\cG_v^\vee)\right).
\end{equation}
Moreover, in $A^*(X)_{\mathbb Q}$,
\begin{equation*}
 \bigl[\operatorname{ch}\Phi_K(\mathfrak K_{Q,\bdelta})\bigr]_j=0
 \quad(j<r_{\bbeta}),
 \qquad
 \bigl[\operatorname{ch}\Phi_K(\mathfrak K_{Q,\bdelta})\bigr]_{r_{\bbeta}}
 =c_{r_{\bbeta}}(\cE_{\bbeta}).
\end{equation*}
\end{proposition}

\begin{proof}
The first identity follows directly from
\eqref{eq:twisted-exterior-class}. For $b:u\to v$ in $Q^{\triangle}$, let
$x_1,\ldots,x_{p_b}$ and $y_1,\ldots,y_{q_b}$ be the Chern roots of
$\cG_u$ and $\cG_v$.  The Chern character of the corresponding factor in
\eqref{eq:associated-bundle-class} is
\[ \prod_{r=1}^{p_b}\prod_{s=1}^{q_b} e^{y_s}(1-e^{x_r-y_s}) =\prod_{r,s}(e^{y_s}-e^{x_r}). \]
Each factor has zero constant term and linear term $y_s-x_r$.  The global product
therefore has no component below codimension $r_{\bbeta}$, and its
component in codimension $r_{\bbeta}$ is
$\prod_{b:u\to v}\prod_{r,s}(y_s-x_r)$, which is the Chern-root expression for
$c_{r_{\bbeta}}(\cE_{\bbeta})$.
\end{proof}

For a tuple of partitions $\bnu=(\nu_v)_{v\in Q_0^{\triangle}}$, with
$\ell(\nu_v)\leq d_v$, set
\[ \Schur_{\bnu}D=\bigotimes_v\Schur_{\nu_v}D_v. \]
For one factor $\GL(D_v)$, the character map identifies
$R_{\mathrm{pol}}(\GL(D_v))$ with the symmetric polynomial ring
$\ZZ[x_{v,1},\ldots,x_{v,d_v}]^{\mathfrak S_{d_v}}$.  Evaluation at the formal
Chern roots of $\cG_v$ therefore defines a ring homomorphism.  Taking the product
over the vertices of $Q^{\triangle}$ gives the \emph{Schur-class evaluation}
\begin{equation}\label{eq:schur-class-evaluation}
 \operatorname{ev}_{\bdelta}:R_{\mathrm{pol}}(G_{\bdelta})\longrightarrow A^*(X),
 \qquad
 [\Schur_{\bnu}D]\longmapsto
 \prod_v s_{\nu_v}(\cG_v).
\end{equation}
Equivalently, it is the unique homomorphism sending the Schur basis to
the corresponding Schur classes.  The Littlewood--Richardson rule verifies
multiplicativity, while the symmetric-polynomial description makes the finite-rank
relations and the independence of a splitting explicit.

\begin{remark}\label{rem:comparison-Schur-K}
The Schur-class map \eqref{eq:schur-class-evaluation} is different from both
$\Phi_K$ and the Chern character: a determinant character is sent to
$c_{d_v}(\cG_v)$, not to $c_1(\det\cG_v)$.  Proposition~\ref{prop:associated-bundle-realization}
shows, however, that the twisted exterior class has an associated-bundle
realization and that its component in codimension $r_{\bbeta}$ is
the top Chern class of the equation bundle.  The Schur-class evaluation is the
degree-preserving homomorphism used in the homogeneous Schur expansion.
\end{remark}

\begin{theorem}[Schur-class evaluation of the twisted exterior class]\label{thm:schur-class-evaluation}
The Schur-class evaluation of the twisted exterior class is the top Chern
class of the universal equation bundle:
\[ \operatorname{ev}_{\bdelta}(\mathfrak K_{Q,\bdelta}) =c_{r_{\bbeta}}(\cE_{\bbeta})\quad\text{in }A^*(X). \]
\end{theorem}

\begin{proof}
Both $\lambda_{-1}$ and $\chi_{\mathrm{det}}$ factor over the arrows of $Q^{\triangle}$.
For $b:u\to v$, Lemma~\ref{lem:finite-rectangular-cauchy} and the Schur-class
evaluation give
\[ \operatorname{ev}_{\bdelta}\!\left( (\det D_v)^{p_b}\lambda_{-1}(D_u\otimes D_v^*)\right) =\prod_{r=1}^{p_b}\prod_{s=1}^{q_b}(u_{v,s}-u_{u,r}), \]
where the $u_{v,s}$ are formal Chern roots of $\cG_v$.  Multiplying over all
$b=(a;i,j)$ gives exactly the Chern-root product
\eqref{eq:top-chern-roots}.
\end{proof}

\subsection{Virtual Schur coefficients}

Since polynomial representations of $G_{\bdelta}$ are semisimple over $\CC$, the
class \eqref{eq:twisted-exterior-class} has a unique finite expansion
\begin{equation}\label{eq:virtual-schur-expansion}
 \mathfrak K_{Q,\bdelta}
 =\sum_{\bnu}\kappa_{\bnu}(Q,\bdelta)
 [\Schur_{\bnu}D].
\end{equation}
Equivalently,
\begin{equation*}
 \kappa_{\bnu}(Q,\bdelta)
 =\sum_{k\geq0}(-1)^k
 \dim\Hom_{G_{\bdelta}}\!\left(
 \Schur_{\bnu}D,
 \chi_{\mathrm{det}}\otimes\Lambda^k((W^{\triangle})^*)
 \right).
\end{equation*}
We call these integers the \emph{virtual Schur coefficients}.  Each is an
alternating sum of the multiplicities of $\Schur_{\bnu}D$ in the polynomial modules
$\mathfrak K^k_{Q,\bdelta}$, and therefore need not be nonnegative.  In particular,
they are not multiplicities in the coordinate ring $\CC[W^{\triangle}]$ and should
not be confused with the positive covariant multiplicities of the one-step theorem.

For a Schubert index $\bw=(w_x)_{x\in Q_0}$ as in Section~6 and a tuple $\bnu$,
define the parabolic Schubert pairing coefficient
\begin{equation}\label{eq:b-w-nu-geometric}
 \mathsf b_{\bw,\bnu}
 =\int_X \sigma_{\bw}\prod_{v=(x,i)}s_{\nu_v}(\cG_x^i).
\end{equation}
It factors vertexwise:
\[ \mathsf b_{\bw,\bnu} =\prod_{x\in Q_0} \int_{X_x} \sigma_{w_x} \prod_{i=1}^{m+1}s_{\nu_x^i}(\cG_x^i). \]
The factor at $x$ is zero unless
$\ell(w_x)+\sum_i|\nu_x^i|=\dim X_x$.  Appendix~\ref{app:divided-difference}
gives a divided-difference evaluation of these factors.

\begin{theorem}[Formula in virtual Schur coefficients]
\label{thm:virtual-schur-Nw}
For every $\bw$ with $\ell(\bw)=d_{\bbeta}$, we have
\[ N_{\bw}(Q,\alpha,\bbeta) =\sum_{\bnu}\mathsf b_{\bw,\bnu} \kappa_{\bnu}(Q,\bdelta). \]
The sum is finite; its terms are determined by $Q^{\triangle}$ and the
parabolic Schubert calculus on $X$.
\end{theorem}

\begin{proof}
Apply the Schur-class evaluation to \eqref{eq:virtual-schur-expansion}, use
Theorem~\ref{thm:schur-class-evaluation}, multiply by $\sigma_{\bw}$, and integrate.
Equation \eqref{eq:Nw-intersection} gives the left-hand side, and
\eqref{eq:b-w-nu-geometric} gives the summands on the right.
\end{proof}

\section{A finite signed formula in auxiliary quiver semi-invariants}
\label{sec:auxiliary-SI}
We expand the twisted exterior class by the Littlewood--Richardson rule
and realize the resulting tensor-product multiplicities as semi-invariant
dimensions of auxiliary star quivers.  The auxiliary quiver is acyclic and
in general disconnected.  It is fixed by $(Q,\bdelta)$; the Cauchy
partitions determine its weight, while the Schubert index enters through
the parabolic pairings.  The formula is signed and does not identify a
Schubert coefficient with a single semi-invariant dimension for $Q$.

\subsection{Littlewood--Richardson expansion}

For each arrow $b:u\to v$ of $Q^{\triangle}$, we choose a partition
$\lambda_b\subset q_b^{p_b}$. Write
\[ \blambda=(\lambda_b)_{b\in Q_1^{\triangle}}, \qquad |\blambda|=\sum_b|\lambda_b|. \]
At a vertex $v$ of $Q^{\triangle}$, put
\[ F_v(\blambda) =\prod_{b:\,t(b)=v}s_{\lambda_b} \prod_{b:\,h(b)=v}s_{\lambda_b^\dagger}. \]
Expand this product in the polynomial character ring of $\GL(D_v)$:
\begin{equation}\label{eq:local-LR-expansion}
 F_v(\blambda)=
 \sum_{\substack{\nu_v\\ \ell(\nu_v)\leq d_v}}
 c_v(\blambda;\nu_v)s_{\nu_v}.
\end{equation}
Terms with more than $d_v$ parts vanish on $D_v$, and
\begin{equation*}
 c_v(\blambda;\nu_v)
 =\dim\Hom_{\GL(D_v)}\!\left(
 \Schur_{\nu_v}D_v,
 \bigotimes_{b:\,t(b)=v}\Schur_{\lambda_b}D_v
 \otimes
 \bigotimes_{b:\,h(b)=v}\Schur_{\lambda_b^\dagger}D_v
 \right).
\end{equation*}
These are generalized Littlewood--Richardson coefficients.

Call a pair $(\blambda,\bnu)$ \emph{admissible} if
\[
 \lambda_b\subseteq q_b^{p_b}\quad(b\in Q_1^{\triangle}),\qquad
 \ell(\nu_v)\leq d_v,\qquad
 |\nu_v|=\sum_{b:t(b)=v}|\lambda_b|
  +\sum_{b:h(b)=v}(p_bq_b-|\lambda_b|)
 \quad(v\in Q_0^{\triangle}).
\]
Let $\mathscr P_{Q,\bdelta}$ be the set of admissible pairs.  It is finite;
for an isolated vertex the condition forces $\nu_v=\varnothing$.

\begin{proposition}\label{prop:Cauchy-LR-Schubert}
For every $\bw$ with $\ell(\bw)=d_{\bbeta}$, we have
\begin{equation}\label{eq:Nw-Cauchy-LR}
 N_{\bw}
 =\sum_{(\blambda,\bnu)\in\mathscr P_{Q,\bdelta}}
 (-1)^{|\blambda|}\mathsf b_{\bw,\bnu}
 \prod_{v\in Q_0^{\triangle}}c_v(\blambda;\nu_v).
\end{equation}
\end{proposition}

\begin{proof}
Apply \eqref{eq:symmetric-rectangular-cauchy} to every factor in the Chern-root
product \eqref{eq:top-chern-roots}. At each vertex of $Q^{\triangle}$, collect the Schur factors
and use \eqref{eq:local-LR-expansion}.  This gives
\[ c_{r_{\bbeta}}(\cE_{\bbeta}) =\sum_{\blambda,\bnu}(-1)^{|\blambda|} \left(\prod_v c_v(\blambda;\nu_v)\right) \prod_v s_{\nu_v}(\cG_v). \]
Multiplication by $\sigma_{\bw}$ and integration yield
\eqref{eq:Nw-Cauchy-LR}.
\end{proof}

\subsection{Star quivers for tensor-product multiplicities}

For a quiver $H$, a dimension vector $\eta$, and a weight
$\sigma\in\ZZ^{H_0}$, the base-change group acts on the coordinate ring by
$(g\cdot f)(M)=f(g^{-1}\cdot M)$.  We use
\[ \SI(H,\eta)_\sigma =\left\{f\in\CC[\Rep(H,\eta)]\ \middle|\ g\cdot f=\left(\prod_{z\in H_0}\det(g_z)^{\sigma(z)}\right)f\right\}. \]
Schofield introduced determinant semi-invariants for prehomogeneous dimension vectors
\cite{SchofieldSI}; Schofield and Van den Bergh proved the spanning theorem for
arbitrary dimension vectors in characteristic zero \cite{SchofieldVdB}.  We use the
convention $\SI(\varnothing,0)_0=\CC$.

Fix $d\geq1$ and $r\geq1$.  Let $\mathcal M_{d,r}$ be the tree with central
vertex $c$, $r$ outward arms
\[ c=z^+_{j,0}\longrightarrow z^+_{j,1}\longrightarrow\cdots \longrightarrow z^+_{j,d}\qquad(1\leq j\leq r), \]
and one inward arm
\[ z^-_d\longrightarrow\cdots\longrightarrow z^-_1 \longrightarrow z^-_0=c. \]
Put
\[ \mathbf d_{d,r}(c)=d, \qquad \mathbf d_{d,r}(z^+_{j,k})=\mathbf d_{d,r}(z^-_k)=d-k+1. \]
The underlying graph is a tree, so $\mathcal M_{d,r}$ is acyclic.

For a partition $\theta=(\theta_1\geq\cdots\geq\theta_d\geq0)$, put
$\theta_{d+1}=0$ and define the outward arm weight
\begin{equation}\label{eq:outward-arm-weight}
 \omega^+_\theta(z^+_{j,k})
 =-(\theta_{d-k+1}-\theta_{d-k+2}).
\end{equation}
For a partition $\nu$ of length at most $d$,
pad $\nu$ with zeroes to length $d$ and set $\nu_{d+1}=0$.
Define the inward arm weight
\begin{equation}\label{eq:inward-arm-weight}
 \omega^-_\nu(z^-_k)
 =\nu_{d-k+1}-\nu_{d-k+2}.
\end{equation}
Given $\theta^1,\ldots,\theta^r$ and $\nu$, let
$\omega(\theta^1,\ldots,\theta^r;\nu)$ be the weight with central component zero,
with $\omega^+_{\theta^j}$ on the $j$th outward arm, and with $\omega^-_\nu$ on
the inward arm.  Telescoping gives
\begin{equation}\label{eq:scalar-weight-condition}
 \omega(\theta^1,\ldots,\theta^r;\nu)\cdot\mathbf d_{d,r}
 =|\nu|-\sum_{j=1}^r|\theta^j|.
\end{equation}

\begin{lemma}[Star-quiver realization of tensor-product multiplicities]\label{lem:star-quiver-tensor}
Let $D$ be $d$-dimensional.  Then
\begin{equation}\label{eq:star-quiver-Hom}
 \dim\SI(\mathcal M_{d,r},\mathbf d_{d,r})_{
 \omega(\theta^1,\ldots,\theta^r;\nu)}
 =\dim\Hom_{\GL(D)}\!\left(
 \Schur_\nu D,\ \bigotimes_{j=1}^r\Schur_{\theta^j}D\right).
\end{equation}
Both sides vanish unless $|\nu|=\sum_j|\theta^j|$.
\end{lemma}

\begin{proof}
Cauchy's formula is applied along every arm.  The prescribed weights on an outward
arm force a unique chain of partitions and leave the central module
$\Schur_{\theta^j}D$.  The opposite orientation and opposite weight differences on
the inward arm leave $\Schur_\nu D^*$.  Appendix~\ref{app:arm-recurrence} gives the
two recurrences.  Taking central weight zero therefore gives
\[ \left(\bigotimes_j\Schur_{\theta^j}D\otimes\Schur_\nu D^*\right)^{\GL(D)}, \]
whose dimension is the right-hand side of \eqref{eq:star-quiver-Hom}.  The weight
vanishes on scalar matrices precisely when \eqref{eq:scalar-weight-condition} is zero.
\end{proof}

\begin{remark}
For $r=2$, the right-hand side of \eqref{eq:star-quiver-Hom} is the ordinary
Littlewood--Richardson coefficient $c_{\theta^1,\theta^2}^{\nu}$.  This star
quiver is the local triple-flag construction of Derksen--Weyman \cite{DW}; it
realizes the same invariant space that appears in
Corollary~\ref{cor:LR-localization}.  For larger $r$ it gives the corresponding
iterated, or generalized, Littlewood--Richardson coefficient.
\end{remark}

\subsection{Semi-invariant realization of the local Littlewood--Richardson coefficients}
Let $r_v$ denote the number of arrows of $Q^{\triangle}$ incident with $v$,
counted with multiplicity, and put
\[ Q_{0,\mathrm{inc}}^{\triangle} =\{v\in Q_0^{\triangle}:r_v>0\}. \]
If $v$ is isolated, then $F_v(\blambda)=1$, the only partition is
$\nu_v=\varnothing$, and $c_v(\blambda;\nu_v)=1$; such a vertex is omitted.
For each $v\in Q_{0,\mathrm{inc}}^{\triangle}$ take a copy of
$\mathcal M_{d_v,r_v}$ and set
\begin{equation}\label{eq:auxiliary-star-quiver}
 (\widehat Q_{\bdelta},\widehat\delta)
 =\coprod_{v\in Q_{0,\mathrm{inc}}^{\triangle}}
 (\mathcal M_{d_v,r_v},\mathbf d_{d_v,r_v}).
\end{equation}
This quiver and dimension vector depend only on $(Q,\bdelta)$.

Given $(\blambda,\bnu)$ in \eqref{eq:Nw-Cauchy-LR}, label the outward arm at $v$
corresponding to an incident arrow $b$ by
\begin{equation*}
 \theta_{v,b}=
 \begin{cases}
  \lambda_b,&t(b)=v,\\
  \lambda_b^\dagger,&h(b)=v.
 \end{cases}
\end{equation*}
Label the inward arm by $\nu_v$, use the weights
\eqref{eq:outward-arm-weight}--\eqref{eq:inward-arm-weight}, and put central weight
zero.  The resulting weight on \eqref{eq:auxiliary-star-quiver} is denoted
$\widehat\sigma(\blambda,\bnu)$.  Its value on scalar matrices in the component at
$v$ is
\[ |\nu_v|-\sum_{b\ni v}|\theta_{v,b}|. \]
It is zero whenever $c_v(\blambda;\nu_v)\neq0$; otherwise both the local tensor
multiplicity and the local semi-invariant space vanish.

\begin{proposition}
\label{prop:LR-as-SI}
For every pair $(\blambda,\bnu)$ occurring in
\eqref{eq:Nw-Cauchy-LR}, we have
\[ \dim\SI(\widehat Q_{\bdelta},\widehat\delta)_{ \widehat\sigma(\blambda,\bnu)} =\prod_{v\in Q_0^{\triangle}}c_v(\blambda;\nu_v). \]
\end{proposition}

\begin{proof}
At the component indexed by $v$, Lemma~\ref{lem:star-quiver-tensor} gives
\[
 \dim\Hom_{\GL(D_v)}\!\left(
 \Schur_{\nu_v}D_v,
 \bigotimes_{b:t(b)=v}\Schur_{\lambda_b}D_v
 \otimes\bigotimes_{b:h(b)=v}\Schur_{\lambda_b^\dagger}D_v
 \right)
 =c_v(\blambda;\nu_v).
\]
The representation space, base-change group, coordinate ring, and weight factor over
the disjoint union, so the global weight space is the tensor product of the local
ones.  Isolated vertices contribute one by convention.
\end{proof}

\begin{theorem}[Semi-invariant formula]\label{thm:auxiliary-SI-Nw}
For every Schubert index $\bw$ with $\ell(\bw)=d_{\bbeta}$,
\[
 N_{\bw}(Q,\alpha,\bbeta)
 =\sum_{(\blambda,\bnu)\in\mathscr P_{Q,\bdelta}}
 (-1)^{|\blambda|}\mathsf b_{\bw,\bnu}\,
 \dim\SI(\widehat Q_{\bdelta},\widehat\delta)_{
 \widehat\sigma(\blambda,\bnu)}.
\]
The auxiliary quiver $(\widehat Q_{\bdelta},\widehat\delta)$ depends only
on $(Q,\bdelta)$; its weight varies with the summation indices.
\end{theorem}

\begin{proof}
Substitute Proposition~\ref{prop:LR-as-SI} into
Proposition~\ref{prop:Cauchy-LR-Schubert}.
\end{proof}

\subsection{Characteristic numbers from auxiliary semi-invariants}

Let $\Xi_d\in A^d(X)$ be any class and write
\[ \Xi_d=\sum_{\ell(\bw)=d}a_{\bw}^{\Xi}\sigma_{\bw}. \]
\begin{corollary}[Characteristic-number formula]\label{cor:characteristic-SI-pairing}
One has
\begin{equation}\label{eq:general-characteristic-SI}
 \int_X c_{r_{\bbeta}}(\cE_{\bbeta})\Xi_d
 =\sum_{\substack{\ell(\bw)=d\\(\blambda,\bnu)\in\mathscr P_{Q,\bdelta}}}
 a_{\bw}^{\Xi}(-1)^{|\blambda|}\mathsf b_{\bw,\bnu}
 \dim\SI(\widehat Q_{\bdelta},\widehat\delta)_{
 \widehat\sigma(\blambda,\bnu)}.
\end{equation}
In particular, for a general quiver flag variety in the dominant case,
\begin{equation*}
 \displaystyle
 \topchi\bigl(\Fl_{\bbeta}(V)\bigr)
 =\sum_{\substack{\ell(\bw)=d_{\bbeta}\\(\blambda,\bnu)\in\mathscr P_{Q,\bdelta}}}
 a_{\bw}(-1)^{|\blambda|}\mathsf b_{\bw,\bnu}
 \dim\SI(\widehat Q_{\bdelta},\widehat\delta)_{
 \widehat\sigma(\blambda,\bnu)}.
\end{equation*}
Every Chern number is a finite integral linear combination of the same auxiliary
semi-invariant dimensions.  Every multiplicative genus is a finite linear combination
over its coefficient ring.
\end{corollary}

\begin{proof}
By Schubert duality and \eqref{eq:flag-cycle-schubert}, the left-hand side of
\eqref{eq:general-characteristic-SI} is
$\sum_{\ell(\bw)=d}a_{\bw}^{\Xi}N_{\bw}$.  Apply
Theorem~\ref{thm:auxiliary-SI-Nw}.  For the Euler characteristic take
$\Xi_d=\Theta_d$ from \eqref{eq:flag-Theta}; the other assertions follow from
Theorem~\ref{thm:flag-euler-integral}.
\end{proof}

\begin{example}
\label{ex:minimal-two-step-cancellation}
Let $Q$ have one arrow $x\to y$, take $\alpha=(3,2)$, and choose
$\beta^1=(1,1)$ and $\beta^2=(2,1)$.  Then
$X=\Fl(1,2;3)\times\mathbb P^1$.  Write $a,b,c$ for the Chern roots of the
three one-dimensional graded pieces at $x$, write $h$ for the last such piece at $y$, and put
\[ A=-a=\sigma_{s_1},\qquad C=c=\sigma_{s_2},\qquad \Sigma_{12}=C^2,\qquad \Sigma_{21}=A^2. \]
Since $a+b+c=0$, the two Chern roots of $\cE_{\bbeta}$ are $h-a$ and $h-b$.  The
Schubert relations on $\Fl(1,2;3)$ give
\[ c_2(\cE_{\bbeta})=(h-a)(h-b)=\Sigma_{12}+hC. \]
Consequently the four Schubert coefficients of complementary degree are
\[
 \begin{array}{c|cccc}
 \bw&(\Sigma_{12},1)&(\Sigma_{21},1)&(A,h)&(C,h)\\ \hline
 N_{\bw}&0&1&1&0.
 \end{array}
\]
Geometrically, for a general map $M:\CC^3\to\CC^2$ with kernel $K$, the flag
condition forces $K\subset P$ and $H=M(P)$; hence the flag variety is the Hirzebruch
surface $\mathbb F_1$.

The two arrows of $Q^{\triangle}$ join one-dimensional vertices.
Their rectangular Cauchy expansion has four terms:
\[ h^2,\qquad -ah,\qquad -bh,\qquad ab. \]
The coefficient paired with $(\Sigma_{12},1)$ receives one positive contribution from
$-ah$ and one negative contribution from $-bh$.  On the auxiliary union
$\mathcal M_{1,1}\sqcup\mathcal M_{1,1}\sqcup\mathcal M_{1,2}$, the
corresponding weights on the local star quivers are
\[
 \begin{aligned}
 \widehat\sigma_{1,0}&:\quad
 (0;-1\mid1),\ (0;0\mid0),\ (0;0,-1\mid1),\\
 \widehat\sigma_{0,1}&:\quad
 (0;0\mid0),\ (0;-1\mid1),\ (0;-1,0\mid1).
 \end{aligned}
\]
Both weight spaces are one-dimensional, and Theorem~\ref{thm:auxiliary-SI-Nw}
therefore gives the cancellation
\begin{equation*}
 N_{(\Sigma_{12},1)}
 =\dim\SI(\widehat Q_{\bdelta},\widehat\delta)_{\widehat\sigma_{1,0}}
  -\dim\SI(\widehat Q_{\bdelta},\widehat\delta)_{\widehat\sigma_{0,1}}
 =1-1=0.
\end{equation*}
\end{example}

\subsection{The one-step specialization}\label{subsec:one-step-stronger}

Assume now that $m=1$ and that $Q$ is acyclic.  Write
$\beta=\beta^1$, $\gamma=\alpha-\beta$, and
$d=\langle\beta,\gamma\rangle_Q$.  For every vertex $x$, let
$R_x$ be the $\beta(x)\times\gamma(x)$ rectangle.  A tuple
$\bmu=(\mu(x))_{x\in Q_0}$ with $\mu(x)\subseteq R_x$ determines a product
Schubert class $\sigma_{\bmu}$; write $\mu(x)^\comp$ for the complementary
partition in $R_x$.  In this notation
\[ c_r(\cE_\beta)=\sum_{|\bmu|=d}N_{\bmu}\,\sigma_{\bmu^\comp}. \]
Let $V_x^\gamma$ be a vector space of dimension $\gamma(x)$,
and put $G_\gamma=\prod_x\GL(V_x^\gamma)$.  Define
\[ \sigma_\beta(x)=\beta(x)-\sum_{a:h(a)=x}\beta(t(a)), \qquad {\textstyle\det^{\sigma_\beta}}=\bigotimes_x(\det V_x^\gamma)^{\sigma_\beta(x)}. \]
Set
\begin{equation}\label{eq:one-step-covariant}
 M_{\bmu}(\beta,\alpha)
 =\dim\Hom_{G_\gamma}\!\left(
 {\textstyle\det^{\sigma_\beta}},\ 
 \CC[\Rep(Q,\gamma)]\otimes
 \bigotimes_x\Schur_{\mu(x)'}V_x^\gamma
 \right).
\end{equation}

\begin{proposition}[Derksen--Schofield--Weyman]\label{prop:one-step-DSW}
For $|\bmu|=d$,
\[ N_{\bmu}=M_{\bmu}(\beta,\alpha). \]
\end{proposition}

\begin{proof}
This is \cite[Proposition~9]{DSW}, translated into the complementary Schubert-basis
convention above.  Its proof compares the Littlewood--Richardson expansion of
Crawley--Boevey's top Chern class with the Cauchy decomposition of
$\CC[\Rep(Q,\gamma)]$ and the multiplicity of the determinant character.
\end{proof}

The covariants can be converted into quiver semi-invariants by adjoining fixed flag
arms to the original quiver.  Attach to each $x\in Q_0$ an outward arm
\[ x=y_{0,x}\longrightarrow y_{1,x}\longrightarrow\cdots \longrightarrow y_{\gamma(x),x} \]
and denote the resulting quiver by $\widehat Q^{\mathrm{os}}$.  Put
\[ \widehat\gamma(x)=\gamma(x), \qquad \widehat\gamma(y_{i,x})=\gamma(x)-i+1. \]
Write
\[ \mu(x)^\comp= \bigl(\gamma(x)^{b_1(x)},(\gamma(x)-1)^{b_2(x)},\ldots, 1^{b_{\gamma(x)}(x)}\bigr) \]
in row-multiplicity notation.  Zero rows are omitted, so
$\sum_j b_j(x)\leq\beta(x)$.  Define
\[
 \widehat\beta_{\bmu}(x)=\beta(x),
 \qquad
 \widehat\beta_{\bmu}(y_{i,x})
 =\sum_{j=1}^{\gamma(x)-i+1}b_j(x),
 \qquad
 \widehat\sigma_{\bmu}=\langle\widehat\beta_{\bmu},\,\cdot\,\rangle.
\]
The original quiver contributes $d$ to the Euler pairing, while the arm at $x$
contributes
\[ -\beta(x)\gamma(x)+ \sum_{i=1}^{\gamma(x)}\widehat\beta_{\bmu}(y_{i,x}) =-|\mu(x)|. \]
Consequently,
$\langle\widehat\beta_{\bmu},\widehat\gamma\rangle=d-|\bmu|=0$ for
$|\bmu|=d$.

\begin{proposition}\label{prop:one-step-flag-arms}
For every $|\bmu|=d$, we have
\[ M_{\bmu}(\beta,\alpha) =\dim\SI(\widehat Q^{\mathrm{os}},\widehat\gamma)_{ \widehat\sigma_{\bmu}}. \]
\end{proposition}

\begin{proof}
This is the flag-arm construction in \cite[Lemmas~10--11]{DSW}.  Successive Cauchy decompositions along the type-$A$
arms impose the Schur factors in \eqref{eq:one-step-covariant}; the Euler-pairing
identity gives weight zero on scalar matrices.
\end{proof}

Expand the one-step virtual tangent class as
\[ \Theta_d(Q,\beta,\gamma) =\left[\frac{c(T_X)}{c(\cE_\beta)}\right]_d =\sum_{|\bmu|=d}a_{\bmu}\sigma_{\bmu}. \]
Combining Proposition~\ref{prop:geometric-pairing} with
Propositions~\ref{prop:one-step-DSW} and
\ref{prop:one-step-flag-arms} gives the following alternative one-step
expression in nonnegative semi-invariant dimensions.  The coefficients
$a_{\bmu}$ of the virtual tangent class can still have either sign:
\[ \displaystyle \topchi\bigl(\Gr_\beta(V)\bigr) =\sum_{|\bmu|=d}a_{\bmu} \dim\SI(\widehat Q^{\mathrm{os}},\widehat\gamma)_{ \widehat\sigma_{\bmu}}. \]
The same substitution gives the other characteristic numbers.  
If $d=0$, then
$\Theta_0=1$, and only $\bmu=\varnothing$ occurs.  Using the original
DSW semi-invariant identification, the formula reduces to their counting
theorem
\[ \#\Gr_\beta(V) =\dim\SI(Q,\gamma)_{\langle\beta,\,\cdot\,\rangle} \]
of Derksen--Schofield--Weyman.  For the triple-flag quiver, the determinant
multiplicity in their proof is the ordinary coefficient
$c_{\lambda,\mu}^{\nu^\vee}$; Corollary~\ref{cor:LR-localization} gives its finite
fixed-point expression.

\section{Examples}\label{sec:examples}

We first study flags for general matrix pencils, followed by examples
for generalized Kronecker quivers.

\begin{center}
\small
\renewcommand{\arraystretch}{1.12}
\begin{tabular}{@{}p{.22\textwidth}|p{.45\textwidth}|c|c@{}}
family & general fiber & $d_{\bbeta}$ & $\topchi$\\ \hline
matrix-pencil flags & $Y_n(T)$; $Y_3(T)$ is del Pezzo of degree $6$
 & $2(n-2)$ & $2^{n-3}n!$\\
$K_2$, $(4,4),(2,3)$ & small resolution of a six-nodal intersection of quadrics
 & $3$ & $12$\\
$K_m$, $(3,m),(1,m-1)$ & degree-$m$ plane curve & $1$ & $m(3-m)$\\
$K_m$, $(m,m),(1,m-1)$ & complete intersection; $K\cong\cO$ & $m-2$
 & \eqref{eq:calabi-yau-euler-coefficient}
\end{tabular}
\end{center}

\subsection{Flags for a general matrix pencil}\label{subsec:matrix-pencil}

Let $K_2$ be the two-arrow Kronecker quiver
\[ x\overset{A,B}{\rightrightarrows}y, \qquad \alpha(x)=\alpha(y)=n, \]
and assume $n\geq3$.  For $1\leq k\leq n-2$, take
\[ \beta^k(x)=k,\qquad \beta^k(y)=k+1. \]
A point of the associated flag variety is a pair of partial flags
\[ F_1\subset F_2\subset\cdots\subset F_{n-2}\subset V_x, \qquad G_2\subset G_3\subset\cdots\subset G_{n-1}\subset V_y, \]
with $\dim F_k=k$, $\dim G_{k+1}=k+1$, and
\[ A(F_k)\subseteq G_{k+1},\qquad B(F_k)\subseteq G_{k+1} \quad(1\leq k\leq n-2). \]
On the dense open set where $A$ is invertible, identify $V_y$ with $V_x=V$ by
$A$ and set $T=A^{-1}B$.  The variety becomes
\[ Y_n(T)= \left\{(F_\bullet,G_\bullet): F_k+T(F_k)\subseteq G_{k+1}\text{ for }1\leq k\leq n-2\right\}. \]
For a general pencil, $T$ is regular semisimple.

\begin{proposition}[A matrix-pencil family]
\label{prop:matrix-pencil-family}
For a general pencil $(A,B)$, the variety $Y_n(T)$ is smooth and projective of
dimension $2(n-2)$.  Let
\[ X_n=\Fl(1,\ldots,n-2;n)\times\Fl(2,\ldots,n-1;n). \]
Write $u_i=c_1(\cG_x^i)$ for $1\leq i\leq n-2$ and
$v_j=c_1(\cG_y^j)$ for $2\leq j\leq n-1$; these are the one-dimensional graded pieces that
occur in the forbidden blocks.  Then
\begin{equation}\label{eq:matrix-pencil-cycle}
 [Y_n(T)]
 =\prod_{\substack{1\leq i\leq n-2\\ i<j\leq n-1}}
 (v_j-u_i)^2
 \quad\text{in }A^*(X_n),
\end{equation}
and
\begin{equation}\label{eq:matrix-pencil-euler}
\displaystyle \topchi(Y_n(T))=2^{n-3}n!.
\end{equation}
\end{proposition}

\begin{proof}
The graded pieces at the source have dimensions $(1,\ldots,1,2)$ and those at the target have
dimensions $(2,1,\ldots,1)$.  Both ambient flag varieties have dimension
$(n-2)(n+1)/2$.  For either arrow, the forbidden blocks are
\[ \cG_x^i\longrightarrow\cG_y^j, \qquad 1\leq i\leq n-2,\quad i<j\leq n-1, \]
so $\rk\cE_{\bbeta}=(n-1)(n-2)$.  
The locus
\[ \mathcal U=\{(A,B):\det A\ne0,\ A^{-1}B\text{ has distinct eigenvalues}\} \]
is a nonempty Zariski-open subset of $\Rep(K_2,(n,n))$.  For every
$(A,B)\in\mathcal U$, identify source and target via $A$, choose an
eigenbasis $e_1,\ldots,e_n$ of $T=A^{-1}B$, and set
\[ F_k=\langle e_1,\ldots,e_k\rangle,\qquad G_{k+1}=\langle e_1,\ldots,e_{k+1}\rangle. \]
Then $F_k+T(F_k)\subseteq G_{k+1}$ for all $k$, so the image of the
incidence projection contains $\mathcal U$ and is dominant.
Theorem~\ref{thm:generic-regular-zero-locus} gives smoothness and dimension
$2(n-2)$, while the block decomposition gives \eqref{eq:matrix-pencil-cycle}.

For the Euler characteristic, choose a diagonal representative of a general regular
semisimple $T$.  The diagonal torus acts on $Y_n(T)$, and its fixed points are pairs of
coordinate chains
\[ I_1\subset\cdots\subset I_{n-2}, \qquad J_2\subset\cdots\subset J_{n-1}, \]
with $|I_k|=k$, $|J_{k+1}|=k+1$, and $I_k\subset J_{k+1}$.  There are $n!/2$
source chains.  For a fixed source chain write
$J_{k+1}=I_k\sqcup\{c_k\}$.  There are two choices for $c_{n-2}$.  If
$I_{k+1}=I_k\sqcup\{a_{k+1}\}$, nestedness gives
$c_k\in\{a_{k+1},c_{k+1}\}$.  Hence each source chain has $2^{n-2}$ compatible
target chains.  A generic one-parameter subgroup has this finite fixed locus, and the
Bia{\l}ynicki--Birula decomposition gives
\[ \topchi(Y_n(T))=\frac{n!}{2}\,2^{n-2}=2^{n-3}n!. \]
\end{proof}

\begin{example}
\label{ex:matrix-pencil-n3}
For $n=3$, the flag type is the one-step type $(1,2)$ and
\begin{equation}\label{eq:pencil-del-pezzo-incidence}
 Y_3(T)=
 \{(L,H)\in\mathbb P(V)\times\mathbb P(V^\vee):L+T(L)\subseteq H\}.
\end{equation}
Away from the eigenlines of $T$, the hyperplane is uniquely determined by
$H=L+T(L)$.  In an eigenbasis with
$T=\operatorname{diag}(\lambda_1,\lambda_2,\lambda_3)$, the resulting rational map is
\[ [x:y:z]\dashrightarrow [ (\lambda_3-\lambda_2)yz: (\lambda_1-\lambda_3)xz: (\lambda_2-\lambda_1)xy]. \]
This is the standard quadratic Cremona transformation up to projective coordinates.
Its base points are the three eigenlines, and \eqref{eq:pencil-del-pezzo-incidence} is
its graph.  Hence
\[ Y_3(T)\cong\operatorname{Bl}_{3\ \mathrm{pts}}\mathbb P^2. \]
It is the degree-six del Pezzo surface, with
$\topchi(Y_3(T))=6$, as in \eqref{eq:matrix-pencil-euler}.
\end{example}

\begin{example}
\label{ex:matrix-pencil-n4}
For $n=4$, the flag type is
\[ \beta^1=(1,2),\qquad \beta^2=(2,3). \]
The variety
\[
 \begin{aligned}
 Y_4(T)=\{(L\subset P,G\subset H):{}&
 L+T(L)\subseteq G,\ P+T(P)\subseteq H\},\\
 &\dim L=1,\quad \dim P=\dim G=2,\quad \dim H=3.
 \end{aligned}
\]
is a smooth fourfold with
\[ [Y_4(T)] =\bigl((v_2-u_1)(v_3-u_1)(v_3-u_2)\bigr)^2, \qquad \topchi(Y_4(T))=48. \]

Forgetting the first source and target steps leads to the one-step quiver
Grassmannian
\[ Z_T= \{(P,H)\in\Gr(2,4)\times\Gr(3,4):P+T(P)\subseteq H\}. \]
Projection to $\Gr(2,4)$ has image
\[ S_T=\{P\in\Gr(2,4):\dim(P+T(P))\leq3\}. \]
Take $T=\operatorname{diag}(\lambda_1,\ldots,\lambda_4)$ and Pl\"ucker coordinates
$p_{ij}$.  The variety $S_T\subset\mathbb P^5$ is cut out by
\begin{align*}
 Q_0&=p_{12}p_{34}-p_{13}p_{24}+p_{14}p_{23},\\
 Q_T&=(\lambda_1\lambda_2+\lambda_3\lambda_4)p_{12}p_{34}
 -(\lambda_1\lambda_3+\lambda_2\lambda_4)p_{13}p_{24}
 +(\lambda_1\lambda_4+\lambda_2\lambda_3)p_{14}p_{23}.
\end{align*}
For general eigenvalues, the three coefficients in $Q_T$ are pairwise distinct.
The Jacobian criterion gives six singular points, the coordinate points $[p_{ij}]$.
On the chart $p_{12}=1$, eliminating $p_{34}$ leaves a nondegenerate quadratic form
in $p_{13},p_{24},p_{14},p_{23}$; by symmetry, all six points are ordinary double
points.

For a non-invariant plane $P\in S_T$, one has $H=P+T(P)$.  Over each invariant
coordinate plane, the possible hyperplanes form $\mathbb P^1$.
The one-step incidence projection with $\alpha=(4,4)$ and
$\beta=(2,3)$ is dominant: for a regular semisimple pencil, take an
invariant coordinate plane and a hyperplane containing it.  The expected
dimension is $4+3-4=3$, so Theorem~\ref{thm:generic-regular-zero-locus}
shows that $Z_T$ is smooth for general $T$. Thus
\[ Z_T\longrightarrow S_T \]
is a projective small resolution with six exceptional curves.  Its diagonal-torus
fixed locus has twelve points, so
\[ \topchi(Z_T)=12. \]

The cycle class of this quiver Grassmannian also displays the
Littlewood--Richardson terms.  Let $x_1,x_2$ be the Chern roots of the universal dual subbundle on
$\Gr(2,4)$ and let $h=c_1(\cQ)$ on $\Gr(3,4)\cong\mathbb P^3$.  Since
$\cE=(\cS^\vee\otimes\cQ)^{\oplus2}$,
\begin{equation*}
 \begin{aligned}
 c_4(\cE)
 &=(x_1+h)^2(x_2+h)^2\\
 &=\sigma_{22}+2h\sigma_{21}
   +h^2(\sigma_2+3\sigma_{11})+2h^3\sigma_1.
 \end{aligned}
\end{equation*}
In particular, the Schubert coefficient paired with
$\sigma_{11}h$ is
\[ N_{(\sigma_{11},h)}=3 =1+1+c_{(1),(1)}^{(1,1)}. \]
The first two summands place $(1,1)$ on one arrow and the empty partition on the
other; the last comes from $s_{(1)}s_{(1)}=s_{(2)}+s_{(1,1)}$.  The
Derksen--Schofield--Weyman theorem identifies the same number with one covariant
multiplicity and with the semi-invariant space obtained from the flag-arm construction
in Subsection~\ref{subsec:one-step-stronger}.
\end{example}

\subsection{Generalized Kronecker families}
We write $K_m$ for the quiver with vertices $1,2$ and $m$ arrows from $1$ to $2$.
For a one-step flag type
\[ \alpha=(n_1,n_2), \qquad \beta=(a,b), \qquad q=n_2-b, \]
the ambient variety, equation bundle, and expected dimension are
\begin{equation*}
 \begin{aligned}
 X&=\Gr(a,n_1)\times\Gr(b,n_2),\qquad
 \cE_{\beta}&=(\cS_1^\vee\otimes\cQ_2)^{\oplus m},\\
 d_{\beta}&=a(n_1-a)+b(n_2-b)-maq.
 \end{aligned}
\end{equation*}
The natural map
$\Rep(K_m,\alpha)\rightarrow H^0(X,\cE_{\beta})$
is an isomorphism when $a>0$ and $q>0$.  Indeed, the standard tautological
bundle identifications give
\[ H^0(\Gr(a,n_1),\cS_1^\vee)=(\CC^{n_1})^\vee,\qquad H^0(\Gr(b,n_2),\cQ_2)=\CC^{n_2}. \]
The degree-zero K\"unneth formula therefore gives
\[ H^0(X,\cS_1^\vee\boxtimes\cQ_2) =\Hom(\CC^{n_1},\CC^{n_2}), \]
and taking $m$ copies gives the asserted isomorphism. These identifications also hold in the boundary cases $a=n_1$ or
$b=0$, provided $a,q>0$. If $a=0$ or
$q=0$, the equation bundle is zero, the natural map is surjective onto
zero, and $Y_\beta(M)=X$ for every representation.  Thus a general
representation gives a general section in every case.

For a formal power series in the variables below, brackets denote coefficient
extraction after all inverse factors have been expanded at the origin.

\begin{proposition}[Chern-root formula for generalized Kronecker quivers]
\label{prop:kronecker-chern-root}
Let $M$ be a general representation of $K_m$ of dimension $\alpha$, and set
$Y_{\beta}(M)=\Gr_{\beta}(M)$.  With the convention
$\topchi(\varnothing)=0$, one has
\begin{align}
\topchi\bigl(Y_{\beta}(M)\bigr)
={}&\frac{1}{a!\,q!}
\left[
 \prod_{i=1}^{a}x_i^{n_1-1}
 \prod_{j=1}^{q}z_j^{n_2-1}
\right]
\prod_{i=1}^{a}(1+x_i)^{n_1}
\prod_{j=1}^{q}(1+z_j)^{n_2}
\notag\\
&\quad\cdot
\prod_{\substack{1\leq i,k\leq a\\ i\neq k}}
 \frac{x_i-x_k}{1+x_i-x_k}
\prod_{\substack{1\leq j,\ell\leq q\\ j\neq\ell}}
 \frac{z_j-z_\ell}{1+z_j-z_\ell}
\prod_{i=1}^{a}\prod_{j=1}^{q}
 \left(\frac{x_i+z_j}{1+x_i+z_j}\right)^m.
\label{eq:kronecker-chern-root}
\end{align}
Empty products are one.  If $d_{\beta}<0$, the indicated coefficient is zero.
\end{proposition}

\begin{proof}
Let $x_1,\ldots,x_a$ be the Chern roots of $\cS_1^\vee$ and let
$z_1,\ldots,z_q$ be those of $\cQ_2$.  The roots of
$\cE_{\beta}=(\cS_1^\vee\otimes\cQ_2)^{\oplus m}$ are the $x_i+z_j$, each with
multiplicity $m$.  In $K^0(X)$,
\[ T\Gr(a,n_1)=n_1\cS_1^\vee-\cS_1^\vee\otimes\cS_1, \qquad T\Gr(b,n_2)=n_2\cQ_2-\cQ_2^\vee\otimes\cQ_2. \]
Hence the integrand in Theorem~\ref{thm:flag-euler-integral} is the top-degree part
of
\begin{equation}\label{eq:kronecker-root-integrand}
 \frac{\prod_i(1+x_i)^{n_1}}{\prod_{i,k}(1+x_i-x_k)}
 \frac{\prod_j(1+z_j)^{n_2}}{\prod_{j,\ell}(1+z_j-z_\ell)}
 \prod_{i,j}\left(\frac{x_i+z_j}{1+x_i+z_j}\right)^m.
\end{equation}
For a symmetric series $f$ in the Chern roots
$\xi_1,\ldots,\xi_k$ of the universal dual subbundle on $\Gr(k,n)$, the standard
coefficient formula is
\begin{equation}\label{eq:grassmann-root-integration}
 \int_{\Gr(k,n)}f(\xi)
 =\frac1{k!}
 [\xi_1^{n-1}\cdots\xi_k^{n-1}]
 \left(f(\xi)\prod_{r\neq s}(\xi_r-\xi_s)\right);
\end{equation}
see, for example, \cite[Section~14.7]{Fulton}.  For the second factor, identify
$\Gr(b,n_2)$ with $\Gr(q,(\CC^{n_2})^\vee)$, so that $\cQ_2$ is the universal dual
subbundle.  Applying \eqref{eq:grassmann-root-integration} to both factors and
substituting \eqref{eq:kronecker-root-integrand} gives
\eqref{eq:kronecker-chern-root}.

A general section of the globally generated bundle $\cE_{\beta}$ is either empty or
regular.  In the empty case its top Chern class vanishes, so the same calculation is
zero.  Finally, degree counting shows that the remaining virtual tangent factor is
taken in degree $d_{\beta}$; hence the coefficient vanishes when $d_{\beta}<0$.
\end{proof}

\begin{example}[A determinantal plane-curve family]
\label{ex:determinantal-plane-curves}
Let $m\geq2$. Take
\[ \alpha=(3,m), \qquad \beta=(1,m-1), \]
and let $M=(M_1,\ldots,M_m)$ be a general representation of $K_m$.  Then
\[ X=\mathbb P^2\times\mathbb P^{m-1}, \qquad \cE_{\beta}\cong\cO(1,1)^{\oplus m}, \qquad d_{\beta}=1. \]
For $0\neq x\in\CC^3$, form the matrix
\[ B_M(x)=\bigl[M_1x\;M_2x\;\cdots\;M_mx\bigr] \in\operatorname{Mat}_{m\times m}(\CC). \]
A point $([x],[y])\in\mathbb P^2\times\mathbb P((\CC^m)^\vee)$ belongs to the
quiver Grassmannian precisely when $yB_M(x)=0$.  Its projection to the first factor
therefore has image
\[ D_M=\{[x]\in\mathbb P^2:\det B_M(x)=0\}. \]
The assignment $M\mapsto B_M$ identifies the representation space with the space
of $m\times m$ matrices of linear forms on $\mathbb P^2$.  For $m\geq2$, a general
projective plane avoids the rank-$\leq m-2$ locus and meets the determinant
hypersurface transversely.  Thus $D_M$ is a smooth plane curve of degree $m$, and the
left kernel gives an isomorphism
\[ \Gr_{(1,m-1)}(M)\cong D_M. \]
Consequently,
\[ g(D_M)=\frac{(m-1)(m-2)}2, \qquad \topchi(D_M)=m(3-m). \]
For the quiver-Grassmannian generating polynomial
\[ F_M(\mathbf y) =\sum_{\mathbf e}\topchi\bigl(\Gr_{\mathbf e}(M)\bigr)\mathbf y^{\mathbf e}, \]
the coefficient of $y_1y_2^{m-1}$ is $m(3-m)$.  This polynomial
enters the Caldero--Chapoton-type cluster character.
A general $M$ need not be rigid, so the
character need not be a cluster variable; negative values for $m\geq4$ do not
contradict cluster-variable positivity.

Let $h,k$ be the hyperplane classes of the two factors.  Since
\[ c_m(\cE_{\beta})=(h+k)^m, \qquad \Theta_1=c_1(T_X)-c_1(\cE_{\beta})=(3-m)h, \]
Theorem~\ref{thm:flag-euler-integral} gives
\[ \topchi(D_M) =(3-m)\int_Xh(h+k)^m =m(3-m). \]
The degree-one Schubert classes are $h$ and $k$.  The coefficient paired with $h$ is
\[ N_h=\int_X(h+k)^m h=m, \]
whereas $N_k=\binom{m}{2}$.  Only $N_h$ enters the Euler pairing because
$\Theta_1=(3-m)h$.  Theorem~\ref{thm:auxiliary-SI-Nw} gives a signed auxiliary semi-invariant
expression for $N_h=m$.  In this one-step acyclic case, the
Derksen--Schofield--Weyman theorem identifies $N_h$ with the corresponding covariant of the
$m$-dimensional arrow space.
\end{example}

\begin{example}[A Calabi--Yau complete-intersection family]
\label{ex:calabi-yau-kronecker}
Let $m\geq2$ and take
\[ \alpha=(m,m), \qquad \beta=(1,m-1). \]
Then
\[ X=\mathbb P^{m-1}\times\mathbb P^{m-1}, \qquad \cE_{\beta}\cong\cO(1,1)^{\oplus m}, \qquad d_{\beta}=m-2. \]
Writing a hyperplane in the second copy of $\CC^m$ as $\ker(y)$, the quiver
Grassmannian is
\[ Y_m= \bigl\{([x],[y])\in X:yM_i x=0\text{ for }1\leq i\leq m\bigr\}. \]
The map from the representation space to
$H^0(X,\cO(1,1)^{\oplus m})$ is an isomorphism.  Thus $Y_m$ is the smooth nonempty
complete intersection of $m$ general divisors of bidegree $(1,1)$.  Adjunction gives
\begin{equation}\label{eq:calabi-yau-adjunction}
 K_{Y_m}
 \cong\bigl(K_X\otimes\det\cE_{\beta}\bigr)|_{Y_m}
 \cong\cO_{Y_m}.
\end{equation}
For $m=2$ the scheme consists of two reduced points.  For $m\geq3$, the Koszul
resolution, Bott vanishing, and K\"unneth give
\[ H^0(Y_m,\cO_{Y_m})=\CC, \qquad H^i(Y_m,\cO_{Y_m})=0\quad(0<i<m-2). \]
Together with \eqref{eq:calabi-yau-adjunction}, this identifies $Y_3$ as an elliptic
curve, $Y_4$ as a K3 surface, and $Y_5$ as a Calabi--Yau threefold.

With $h,k$ denoting the two hyperplane classes, the Euler characteristic is
\begin{equation}\label{eq:calabi-yau-euler-coefficient}
 \topchi(Y_m)
 =[h^{m-1}k^{m-1}]
 (h+k)^m\frac{(1+h)^m(1+k)^m}{(1+h+k)^m}.
\end{equation}
Write
\[ \Theta_{m-2} =\left[\frac{(1+h)^m(1+k)^m}{(1+h+k)^m}\right]_{m-2} =\sum_{r=0}^{m-2}a_{m,r}h^rk^{m-2-r}. \]
The Schubert coefficient paired with $h^rk^{m-2-r}$ is
\[ N_r =\int_X(h+k)^m h^rk^{m-2-r} =\binom{m}{r+1}. \]
Therefore the geometric pairing of Proposition~\ref{prop:geometric-pairing} becomes
\begin{equation}\label{eq:calabi-yau-schubert-pairing}
 \topchi(Y_m)=\sum_{r=0}^{m-2}a_{m,r}\binom{m}{r+1}.
\end{equation}
If $A\cong\CC^m$ is the span of the arrows, the one-step
Derksen--Schofield--Weyman theorem identifies $N_r$ with the dimension of the
covariant $\bigwedge^{m-1-r}A$. Indeed, with $U=\CC^{m-1}$ and
$L=\CC$, the coordinate ring is $\Sym(U\otimes L^\vee\otimes A)$.
The required multiplicity
$\Hom_{\GL(U)}(\det U,\Schur_\lambda U\otimes\bigwedge^rU)$
is one for $\lambda=(1^{m-1-r})$ and zero for other partitions of that
degree.  Cauchy's formula then gives $\bigwedge^{m-1-r}A$. Thus
\eqref{eq:calabi-yau-schubert-pairing} is the dimension of the virtual
$\GL(A)$-module
\[ \sum_{r=0}^{m-2}a_{m,r} [\textstyle\bigwedge^{m-1-r}A]. \]
For $m=2,3,4,5$ one finds
\[ \Theta_{m-2}=1,\qquad 0,\qquad 4hk,\qquad -5h^2k-5hk^2, \]
respectively, and hence
\[ \topchi(Y_m)=2,\qquad0,\qquad24,\qquad-100. \]
These are the Euler characteristics of two points, an elliptic curve, a K3 surface,
and a Calabi--Yau threefold.
\end{example}

\appendix

\section{Divided-difference computation of the parabolic Schubert pairings}\label{app:divided-difference}

We give a divided-difference formula for \eqref{eq:b-w-nu-geometric}.  Fix a vertex
$x$, put $n=\alpha(x)$, and set
\[ c_i=\delta^1(x)+\cdots+\delta^i(x), \qquad c_0=0. \]
Let $P_x\subset\GL_n$ be the parabolic with block sizes
$\delta^1(x),\ldots,\delta^{m+1}(x)$, let $W_x=S_n$, and let
\[ W_{P_x}=S_{\delta^1(x)}\times\cdots\times S_{\delta^{m+1}(x)}. \]
Write $W_x^{P_x}$ for the minimal right-coset representatives and
$w_{0,x}^{P_x}=w_{0,x}w_{0,P_x}$ for the longest element of $W_x^{P_x}$.
For variables $z_1,\ldots,z_n$, define
\[ \partial_i f=\frac{f-s_i f}{z_i-z_{i+1}}. \]
For a reduced expression $w=s_{i_1}\cdots s_{i_\ell}$, use the convention
\[ \partial_w=\partial_{i_1}\circ\cdots\circ\partial_{i_\ell}, \]
where functional composition is read from right to left.  With the standard
right-descent convention for Schubert polynomials this gives
$\partial_w\mathfrak S_w=1$.  Put
\[ Z_x^i=(z_{c_{i-1}+1},\ldots,z_{c_i}). \]
Choose the Schubert convention in which, after pullback to the complete flag variety,
$\sigma_{w_x}$ is represented by $\mathfrak S_{w_x}(z)$ and
$z_j=-c_1(\cL_j)$ for the universal successive line quotients.  Then the Chern roots
of $\cG_x^i$ are the negatives of the variables in $Z_x^i$.  Throughout,
$s_\nu(-Z_x^i)$ means literal substitution
$z\mapsto-z$ in the Schur polynomial; in particular,
$s_\nu(-Z_x^i)=(-1)^{|\nu|}s_\nu(Z_x^i)$.  It does not denote the supersymmetric
Schur function of a virtual negative alphabet.

\begin{proposition}\label{prop:b-divided-difference}
If $\ell(w_x)+\sum_i|\nu_x^i|=\dim X_x$, then
\[ \int_{X_x} \sigma_{w_x}\prod_i s_{\nu_x^i}(\cG_x^i) =\partial_{w_{0,x}^{P_x}}\!\left( \mathfrak S_{w_x}(z) \prod_i s_{\nu_x^i}(-Z_x^i) \right). \]
If the degree condition fails, the integral is zero.  In particular,
$\mathsf b_{\bw,\bnu}$ is computable by a product of finite divided-difference
calculations.
\end{proposition}

\begin{proof}
Because $w_x$ is a minimal right-coset representative, it has no right descents in
$W_{P_x}$.  Hence $\partial_i\mathfrak S_{w_x}=0$ for every simple reflection in
$W_{P_x}$, so $\mathfrak S_{w_x}$ is symmetric in the variables of each parabolic
block.  The Schur factors are block-symmetric as well.  The standard Borel
presentation and the pushforward formula for $\GL_n/P_x$ therefore identify
integration of the top-degree $W_{P_x}$-invariant polynomial with
$\partial_{w_{0,x}^{P_x}}$.  The operator lowers the degree by
$\ell(w_{0,x}^{P_x})=\dim X_x$, so under the stated degree hypothesis its value is
a constant.  The sign in $s_{\nu_x^i}(-Z_x^i)$ records the choice
$z_j=-c_1(\cL_j)$.  See, for example, \cite[Chapter~3]{Manivel}.
\end{proof}

\section{Partition recurrences along the flag arms}\label{app:arm-recurrence}

This appendix gives the partition calculation used in
Lemma~\ref{lem:star-quiver-tensor}.  On an outward arm put $E_0=D$,
$\dim E_k=d-k+1$, and attach a partition $\lambda^{+,k}$
to the arrow $E_{k-1}\to E_k$.  Cauchy's formula gives the factor
\[ \Schur_{\lambda^{+,k}}E_{k-1} \otimes\Schur_{\lambda^{+,k}}E_k^*. \]
At an internal vertex $E_k$ of dimension $e=d-k+1$, taking determinant weight
$\omega^+_\theta(z_k)$ forces
\begin{equation}\label{eq:outward-padded-recurrence}
 (\lambda^{+,k+1}_1,\ldots,\lambda^{+,k+1}_{e-1},0)
 =(\lambda^{+,k}_1,\ldots,\lambda^{+,k}_e)
 +\omega^+_\theta(z_k)(1^e),
\end{equation}
and at the terminal vertex
$\omega^+_\theta(z_d)=-\lambda^{+,d}_1$.

On the inward arm, attach $\lambda^{-,k}$ to the arrow $E_k\to E_{k-1}$. The
factor at the central vertex is dual. At an internal vertex, taking weight
$\omega^-_\nu(z_k)$ forces
\begin{equation}\label{eq:inward-padded-recurrence}
 (\lambda^{-,k}_1,\ldots,\lambda^{-,k}_e)
 =(\lambda^{-,k+1}_1,\ldots,\lambda^{-,k+1}_{e-1},0)
 +\omega^-_\nu(z_k)(1^e),
\end{equation}
and the terminal condition is
$\omega^-_\nu(z_d)=\lambda^{-,d}_1$.

\begin{lemma}\label{lem:arm-recurrence-solutions}
The outward recurrence has the unique solution
\begin{equation}\label{eq:outward-unique-partitions}
 \lambda^{+,k}
 =\bigl(\theta_1-\theta_{d-k+2},\ldots,
 \theta_{d-k+1}-\theta_{d-k+2}\bigr),
\end{equation}
while the inward recurrence has the unique solution
\begin{equation}\label{eq:inward-unique-partitions}
 \lambda^{-,k}
 =\bigl(\nu_1-\nu_{d-k+2},\ldots,
 \nu_{d-k+1}-\nu_{d-k+2}\bigr).
\end{equation}
In particular, $\lambda^{+,1}=\theta$ and $\lambda^{-,1}=\nu$.
\end{lemma}

\begin{proof}
For the outward arm, the terminal condition gives
$\lambda^{+,d}=(\theta_1-\theta_2)$, and backward induction in
\eqref{eq:outward-padded-recurrence} gives
\eqref{eq:outward-unique-partitions}.  For the inward arm, the terminal condition gives
$\lambda^{-,d}=(\nu_1-\nu_2)$, and
\eqref{eq:inward-padded-recurrence} gives
\eqref{eq:inward-unique-partitions} by the same backward induction.
\end{proof}

After taking all noncentral weights, an outward arm leaves
$\Schur_\theta D$, whereas the inward arm leaves $\Schur_\nu D^*$.  This proves
the star-quiver calculation used in Lemma~\ref{lem:star-quiver-tensor}.

\end{document}